\documentclass[a4,10pt]{article}%
\usepackage{amsmath}%

\usepackage{tikz}
\usepackage{amsfonts}%
\usepackage{amssymb}%
\usepackage{graphicx}
\usepackage{pstricks}
\usepackage{pstricks-add}
\usepackage[utf8]{inputenc}
\usepackage{pgf,tikz}
\usetikzlibrary{arrows}

\newtheorem{theorem}{Theorem}

\newtheorem{claim}[theorem]{Claim}

\newtheorem{corollary}[theorem]{Corollary}

\newtheorem{definition}[theorem]{Definition}

\newtheorem{lemma}[theorem]{Lemma}

\newtheorem{proposition}[theorem]{Proposition}

\newcommand{\calP}{{\cal P}}

\newcommand{\ep}{\varepsilon}
\newcommand{\dN}{{{\bf N}}}
\newcommand{\dR}{{{\bf R}}}
\newcommand{\E}{{{\bf E}}}

\newcommand{\prob}{{{\bf P}}}

\newcommand{\cone}{{\rm{cone}}}

\newcommand{\calC}{{\cal C}}

\newcommand{\calO}{{\mathcal{O}}}
\newcommand{\calF}{\mathcal{F}}
\newcommand{\calE}{\mathcal{E}}
\newcommand{\calL}{\mathcal{L}}

\newcommand{\calS}{\mathcal{S}}

\newenvironment{proof}[1][Proof]{\textbf{#1.} }{\ \rule{0.5em}{0.5em}}
\newenvironment{proofclaim}[1][Proof]{\textbf{#1.} }{\ \rule{0.5em}{0.5em}}

\newcounter{figurecounter}
\begin{document}

\title{Optimal Dynamic Information Provision%
\thanks{The research of  Renault and Vieille
was supported by Agence Nationale de la Recherche (grant
ANR-10-BLAN 0112). Solan acknowledges the support of the Israel
Science Foundation, Grants \#212/09 and \#323/13.
The authors thank Omri Solan for numerics which have led to  the counterexample presented in Section \ref{sec_counter}.}}

\author{J\'er\^ome Renault\thanks{TSE (GREMAQ, Universit\' e Toulouse 1),
  21 all\' ee de Brienne, 31000 Toulouse, France. E-mail: \textsf{jerome.renault@tse-fr.eu}.},
Eilon Solan\thanks{School of Mathematical Sciences, Tel Aviv University, Tel
Aviv 69978, Israel. E-mail: \textsf{eilons@post.tau.ac.il}.},
and Nicolas Vieille\thanks{Departement Economics and Decision Sciences, HEC Paris, 1, rue de
la Lib\'{e}ration, 78 351 Jouy-en-Josas, France. E-mail: \textsf{vieille@hec.fr}.}}

\date{\today}

\maketitle

\begin{abstract}
We study a dynamic model of information provision.
A state of nature evolves according to a Markov chain. An informed advisor decides how much information to provide to an uninformed decision maker,
so as to influence his short-term decisions.
We deal with a stylized class of situations, in which the decision maker has a risky action and a safe action,
and the payoff to the advisor only depends on the action chosen by the decision maker.
The greedy disclosure policy is the policy which, at each round, minimizes the amount of information being disclosed in that round, under the constraint that it maximizes the current payoff of the advisor. We prove that the greedy policy is optimal in many cases -- but not always.
\end{abstract}


\section{Introduction}

Market conditions evolve over time, and information that is privately available to a market participant
is a valuable asset. In this paper we study
the optimal provision of information by an informed ``expert'' with no decision power, to an uninformed agent in  a dynamic setup. We develop a stylized model
in which an ``investor" chooses at each date whether or not to
choose a risky action, such as a short-run investment. The payoff
from investing depends on some underlying state of nature, which
is unknown to the investor. This state accounts for all relevant
external factors and evolves exogenously according to a Markov
chain.

At each date, the investor may get information through the
advisor. How much information is being disclosed is the choice
variable of the advisor.
To be specific, the advisor publicly chooses an
information provision rule, which maps each history into a
distribution over signals. The investor observes both the rule
chosen by the advisor and the realized signal. We assume that the
advisor receives a fixed fee whenever investment takes place, and
that the investor invests whenever the expected net payoff given
his current posterior belief is nonnegative.

This allows us to recast the problem faced by the advisor as a
Markov decision problem (MDP) in which the state space is the compact
set of posterior beliefs of the investor, and the action space is
the set of information provision rules. In that MDP, the advisor
chooses dynamically the provision of information so as to maximize
the (expected) discounted frequency of dates in which investment
takes place. Advising is thus both honest, in that realized
signals cannot be manipulated, and strategic, in that the
information content of the signal is strategic.

There are two (mutually exclusive) interpretations that befit this
description. In the first one, the advisor does not observe the
underlying state, and chooses how much information will be
publicly  obtained. In other words, he chooses a statistical
experiment \textit{\`a la} Blackwell, whose outcome is public. In the
second interpretation, the advisor does observe the successive states of nature but commits \textit{ex ante} to a dynamic information provision
policy.

The basic trade-off faced by the advisor is the following. By
disclosing information at a given date, the advisor may increase
his payoff at that date, but then gives up  part of his
information advantage for later dates, as soon as successive
states are autocorrelated. Our focus is on structural properties
of the model. Characterizing optimal information provision in
general is out-of-reach, and we instead focus on the optimality of
rules in which the above trade-off is solved in a very simple way.
We define the greedy policy as the one that, at any given date,
minimizes the amount of information being disclosed, subject to
the current payoff of the advisor being maximized. We prove that
this policy is optimal in the case of two states of nature. We
then exhibit a class of Markov chains, described by a renewal
property,  for which this policy is optimal for a large range of
initial distributions of the state (including most natural ones),
and is eventually optimal, for any  initial distribution of the
state. Our main message is thus  that this policy is likely to
perform very well in a large number of cases, but not always, as
we show by means of a counterexample.

Our modelling of information acquisition/disclosure is a dynamic
version of the persuasion mechanisms of Kamenica and Gentskow
(2011) who study optimal signals in a broader, yet static, setup.
It also parallels the independent paper by Ely (2014). Our paper
joins the growing literature on dynamic  models in which
uncertainty evolves, see, e.g., Mailath and Samuelson (2001),
Phelan (2006), Wiseman (2008), or Athey and Bagwell (2008), and
Escobar and Toikka (2013) for economic applications. These
references focus on game models, whose mathematical analysis is in
general quite challenging, see Renault (2006) and H\"orner et
al.~(2010). Although our basic model is a game-theoretic one, its
reduced form, and the commitment assumption makes it more
comparable to contract theory ones, see e.g. Battaglini (2005),
Zhang and Zenios (2008) or Zhang, Nagarajan and Sosic (2008).


\section{Model and Main Results}
\label{sec:model}

\subsection{Model}

We consider the following stylized class of two-player games
between an ``advisor'' (Player~1) and an ``investor'' (Player~2).
The advisor observes a stochastic process $(\omega_n)_{n\in \dN}$
with values in a finite set of states $\Omega$, and may provide
the investor with information regarding the current or past values of the process. In each round, the investor chooses
whether to invest or not. The investor's payoff from investing in round $n$
is $r(\omega_n)$, where $r: \Omega\to \dR$. The advisor receives a
fee  whenever investment takes place (this fee is already
accounted for in $r$) and discounts future payoffs according to a
discount factor $\delta$.

While the investor knows the law of the sequence
$(\omega_n)_{n\in\dN}$, he receives no information on the realized
states, except through the advisor. It is then natural to assume
that he chooses to invest whenever his expected (net) payoff from
investing is nonnegative, where the expectation is computed using
the information released by the advisor.\footnote{From the
literature on dynamic games we know that more sophisticated
equilibria may possibly be designed. Besides being natural, our
assumption allows to cover the case of short-lived investors or
of a large number of investors.}$^{,}$\footnote{To simplify the analysis we will assume that the investor also invests on the investment frontier,
that is, when his expected profit is 0.
Indeed, otherwise, whenever the investor's belief is on the investment frontier,
the advisor would reveal a small amount of additional information, so as to push the investor's belief to the region
where the investor strictly prefers investing to not investing.}
Thus, the game reduces to a
stochastic optimization problem, in which the advisor chooses
whether and how to reveal information to the investor, so as to
maximize the expected discounted frequency of rounds in which
investment takes place.

We assume that the process $(\omega_n)_{n\in\dN}$ follows an irreducible Markov
chain with transition matrix $M=(\pi(\omega'\mid \omega))_{\omega,\omega' \in \Omega}$ and
invariant measure $m\in \Delta(\Omega)$. The set $\Delta(\Omega)$
is the set of probability distributions over $\Omega$, whose
elements are potential ``beliefs'' of the investor.
Throughout, we identify each $\omega\in \Omega$ with a unit basis vector in $\dR^\Omega$, and $\Delta(\Omega)$ with the $(|\Omega| -1)$-dimensional unit simplex in $\dR^\Omega$, endowed with the induced topology.

The game is played as follows.
In each round
$n$, the state $\omega_n$ is drawn according to $\pi(\cdot\mid
\omega_{n-1})$, the advisor observes $\omega_n$ and chooses which
message to send to the investor;
The investor next chooses whether
to invest, and the game moves to the next round.\footnote{We are
not explicit about the message set. It will be convenient to first
assume that it is rich enough, e.g., equal to $\Delta(\Omega)$. We
will show that w.l.o.g. two messages suffice.}
When the investor's
belief is $p\in \Delta(\Omega)$, his expected net payoff from
investing is given by the scalar product $\langle p,r\rangle = \sum_{\omega \in \Omega} p(\omega) r(\omega)$.
Accordingly, the \emph{investment region} is $I:=\{p\in
\Delta(\Omega), \langle p,r\rangle\geq 0\}$ and the \emph{investment
frontier} is $\calF:=\{p\in \Delta(\Omega), \langle p,r\rangle= 0\}$. We also denote by $J:=\Delta(\Omega)\setminus I$ the \textit{noninvestment} region.

Throughout, we will denote by $\Omega^+:=\{\omega\in \Omega, r(\omega)\geq 0\}$ and $\Omega^-:=\{\omega\in \Omega, r(\omega)<0\}$ the states with nonnegative and negative payoff respectively, so that $\Omega^+$ and $\Omega^-$ form a partition of $\Omega$.

An \emph{information disclosure policy} for the advisor specifies for each round,
the probability law of the message being sent in that round, as a
function of previous messages and the information privately
available to the advisor, that is, past and current states.

We will assume that the advisor has \emph{commitment power}. To be specific, we assume that in any given round, the investor knows which disclosure policy was used in that round, and therefore knows unambiguously how to interpret the message received from the advisor.

An equivalent and alternative interpretation
is to assume that the advisor does
\emph{not} observe the process $(\omega_n)_{n\in\dN}$ and chooses
in each round a statistical experiment \textit{\`a la} Blackwell. Such an
experiment yields a random outcome, whose distribution is
contingent on the current state. Under this alternative
interpretation,
the advisor has no private
information, but by choosing the experiment, he effectively determines how much information is being publicly
obtained,
and
the investor observes both the experiment choice  and the outcome of the experiment.

\subsection{A Reformulation}

Given an information disclosure policy, the investor uses the
successive messages received from the advisor  to update his
belief on the current state. We find it convenient to distinguish
the beliefs $p_n$  and $q_n$ held in round $n$, respectively
\emph{before} and \emph{after} receiving the message of the advisor. Formally, $p_n$
is the conditional law of $\omega_n$ given the messages received
prior to round $n$, while $q_n$ is the updated belief, once the
round $n$ message has been received,  so that the investor invests in round $n$ if and only if
$q_n\in I$.

The beliefs $q_n$ and $p_{n+1}$ differ because
the state evolves:
 $\omega_n$ and
$\omega_{n+1}$ need not be equal, and one has $p_{n+1}=
\phi(q_n):=q_nM$.
The difference between
$p_n$ and $q_n$ is the result of the information provided by the advisor.

For a given
$p\in \Delta(\Omega)$, denote by $\calS(p)\subset
\Delta(\Delta(\Omega))$ the set of probability distributions over
$\Delta(\Omega)$ with mean $p$.
We denote by $\mu_p \in \calS(p)$ the distribution over $\Delta(\Omega)$ that assigns probability 1 to $p$.

As a consequence of Bayesian
updating, the (conditional) law $\mu$ of $q_n$  belongs to
$\calS(p_n)$, for every information disclosure policy. Conversely,
a classical result from the literature of repeated games
with incomplete information (see Aumann and Maschler (1995)\footnote{Aumann and Maschler (1995) contains a
proof when the distribution $\mu$ has a finite support. Their
proof readily extends to the case in which the support of $\mu$ is general.})
states that the converse also holds.
That is, given any distribution $p\in \Delta(\Omega)$ and any
distribution $\mu\in \calS(p)$ of beliefs with mean $p$, the
advisor can correlate the message with the state in such a way
that the investor's updated belief is distributed according
to $\mu$. Elements of $\calS(p)$ will be called \emph{splittings at $p$}, as is common in the literature.%
\footnote{Or simply \emph{splitting}, if $p$ is clear from the context.}

These observations allow us to reformulate the decision problem faced
by the advisor as a dynamic optimization problem
$\Gamma$. The state space in $\Gamma$ is the set $\Delta(\Omega)$ of investor's beliefs
and the initial state is $p_1$, the law of $\omega_1$. At each state
$p\in \Delta(\Omega)$, the set of available actions is the set
$\calS(p)$, so that the advisor chooses a distribution $\mu$ of
posterior beliefs that is consistent with $p$.  Given the
posterior belief $q$, the current payoff is 1 if $q\in I$ and
0 if $q\notin I$, and the next state in $\Gamma$ is $\phi(q)$.
Thus, the (expected) stage payoff given $\mu$ is $\mu(q\in I)$.

We denote by $V_\delta(p_1)$ the value of $\Gamma$ as a function
of the initial distribution $p_1$.
The value function $V_\delta$ is characterized as
the unique solution of the dynamic programming equation\footnote{We write
$\max$ on the right-hand side because it is readily checked that
$V_\delta$ is Lipschitz over $\Delta(\Omega)$, the expression
between braces is upper hemi-continuous w.r.t.~$\mu$ in the
weak-* topology on $\Delta(\Delta(\Omega))$, and $\calS(p)$ is
compact in that topology. Details are standard and omitted.}
\begin{equation}
\label{equ1}V_\delta(p)= \max_{ \mu\in S(p)} \left\{ (1-\delta)  \mu(q\in I) +
\delta \E_\mu\left[ V_\delta(\phi(q))\right] \right\}, \ \ \ \forall p \in \Delta(\Omega).
\end{equation}

\subsection{The (static) value of information}

We first argue that the value function $V_\delta$ is concave.
This result
has a number of implications on the structure of the advisor's optimal strategy.
We will  point at two such implications which
are especially useful in the sequel.

\begin{lemma}\label{lemm1}
The function $V_\delta$ is concave on $\Delta(\Omega)$.
\end{lemma}

\begin{proof}
This is a standard result in the literature on zero-sum games with
incomplete information, see, e.g., Sorin (2002, Proposition 2.2). While
the setup here is different, the proof follows the same logic, and
we only sketch it. We need to prove that $V_\delta(p)\geq
a'V_\delta(p')+a''V_\delta(p'')$ whenever $p=a'p'+a'' p''$,
with $a',a''\geq 0$ and $a'+a''=1$. Starting from $p$,  consider
the following strategy $\sigma$ for the advisor.  Pick first the
element $\mu\in \calS(p)$ that assigns probabilities $a'$ and
$a''$ to $p'$ and $p''$ respectively, and next follow an optimal strategy in
$\Gamma(p')$ or $\Gamma(p'')$, depending on the outcome of $\mu$.
Thus, the advisor's behavior at $p$ is
a so-called \emph{compound lottery} obtained as the result of first using $\mu$, and then the first choice of an optimal strategy
in either $\Gamma(p')$ or $\Gamma(p'')$.

The strategy $\sigma$
yields at $p$ a payoff equal to $a'V_\delta(p')+a''V_\delta(p'')$,
hence the result.
\end{proof}

\bigskip

The first consequence of Lemma \ref{lemm1} is that the advisor does not reveal information when the investor's belief is in the investment region.

\begin{corollary}\label{cor1}
At any $p\in I$, it is optimal for the advisor not to provide information to the investor.
\end{corollary}

That is, the distribution $\mu_p\in \calS(p)$ that assigns
probability one to $p$ achieves the maximum in (\ref{equ1}).

The intuition is as follows.
When $p \in I$, revealing information cannot increase the current payoff,
and therefore, such revelation may only possibly be beneficial in subsequent stages.
However, every information that is disclosed today could instead be revealed tomorrow,
so that there is no reason to provide information to the investor when $p \in I$.
Note that we do not rule out the possibility that there are \emph{additional} optimal strategies that do reveal information in $I$.

\bigskip

\begin{proof} Fix $\mu\in \calS(p)$.
By the concavity of the function $q \mapsto V_\delta(\phi(q))$ and Jensen's inequality, one has
\[\E_\mu\left[ V_\delta(\phi(q))\right]\leq V_\delta(\phi(\E_\mu[q]))=V_\delta(\phi(p)),\]
with equality for $\mu=\mu_p$.
Moreover, $\mu(q\in I)$ cannot exceed 1, and is equal
to 1 for $\mu=\mu_p$.
Therefore the right-hand side in (\ref{equ1}) is at most
$(1-\delta)+\delta V_\delta(\phi(p))$,
and this upper bound is achieved for $\mu=\mu_p$.\end{proof}
\bigskip

A second corollary of Lemma \ref{lemm1} states that in the investment region,
the advisor can restrict himself to splitting the investor's belief among at most two beliefs.

\begin{corollary}\label{cor2}
At any $p\notin I$, there is an optimal choice $\mu\in \calS(p)$,
which is carried by at most two points.
\end{corollary}

That is, at each $p\notin I$ it is either optimal not to disclose
information, or to disclose information in a coarse way so that
the posterior belief of the investor  takes only two well-chosen
values in $\Delta(\Omega)$.
This result hinges on the fact that (i) the advisor's stage payoff assumes  two values only,
and (ii) the investment region $I$ is convex.

\bigskip

\begin{proof}
Let $p\notin I$ and $\mu\in \calS(p)$ be arbitrary.
Assume first that
$\mu(q\in I)=0$ and compare the distribution $\mu$ to the distribution $\mu_p$
in which no information is revealed.
The two distributions yield the same current
payoff, because $\mu(q\in I)=\mu_p(q\in I)=0$.
However, $\mu_p$ yields a (weakly) higher continuation
payoff, because by Jensen's inequality
\[\E_{\mu_p}\left[ V_\delta (\phi(q))\right]= V_\delta (\phi(p))\geq \E_{\mu}\left[ V_\delta (\phi(q))\right].\]
Assume now that $\mu(q\in I)>0$. Since $p\in J$ and $I$ is convex, one also has $\mu(q\in J)>0$.

Denote by $q_I:=\E_\mu\left[ q\mid q\in I\right]$ (resp.~$q_J:=\E_\mu\left[ q\mid q\in J\right]$)
the expected posterior belief conditional on it being in (resp.~not in) the investment region.
Then
\[p=\mu(q\in I) q_I +\mu(q\in J) q_J.\]
Denote by $\tilde \mu\in \calS(p)$ the two-point distribution that assigns
probabilities $\mu(q\in I)$ and $\mu(q\in J)$ to $q_I$ and $q_J$
respectively. Plainly,  $\tilde \mu(q\in I)=\mu(q\in I)$ and
\[\E_{\tilde \mu}\left[ V_\delta(\phi(q))\right]= \mu(q\in I) V_\delta(\phi(q_I))+\mu(q\in J) V_\delta(\phi(q_J)),\]
while
\begin{eqnarray*}\E_{ \mu}\left[ V_\delta (\phi(q))\right]&=&
\mu(q\in I) \E_{ \mu}\left[ V_\delta (\phi(q))\mid q\in I\right]
+\mu(q\in J) \E_{ \mu}\left[ V_\delta (\phi(q))\mid q\in J\right] \\
&\leq & \mu(q\in I) V_\delta (\phi\left(\E_{ \mu}\left[ q\mid q\in I\right]\right))
+\mu(q\in J) V_\delta (\phi\left(\E_{ \mu}\left[ q\mid q\in J\right]\right)) \\
&\leq & \E_{\tilde \mu}\left[ V_\delta (\phi(q))\right].
\end{eqnarray*}
To sum up, for any given $\mu$,  we have shown that either the no disclosure policy
$\mu_p$, or some two-point distribution $\tilde \mu$ yields a weakly higher
right-hand side in (\ref{equ1}) than
$\mu$. This proves the result.
\end{proof}

\bigskip

Note that it may still be optimal not to disclose information at $p\in J$. This is in particular the case whenever $p(\Omega^+)=0$.

\subsection{Main Results}

The intuition behind Corollaries \ref{cor1} and \ref{cor2} above is
clear. When $p\in I$, no information should be revealed, because it
cannot help to increase the current payoff, and can only hurt
continuation values. When $p\notin I$, there are two conflicting
effects at play. For the sake of maximizing payoffs, the advisor
should release information. But releasing information may only
hurt continuation payoffs, because of concavity.

Corollary \ref{cor2} shows qualitatively (but not explicitly) how
to compromise between the two effects. The main message of our results is that in many
cases but not all, the explicit compromise is simple: the advisor
should minimize the amount of information released, subject to
current payoffs being maximal.
We define accordingly the \emph{greedy strategy} $\sigma_*$ as
follows.

\begin{definition}\label{def_greedy}
The \emph{greedy strategy} for the advisor is the strategy $\sigma_*$
that depends on the investor's current belief $p$, and plays as follows:
\begin{description}
\item[G1]
At any $p\in I$, the strategy $\sigma_*$ discloses no information.
\item[G2]
At any $p\notin I$, the strategy $\sigma_*$ chooses a solution $\tilde \mu\in \calS(p)$ to
the problem $\max a_I$, under the constraints $p=a_I q_I+a_J q_J$,
$q_I\in I$, $a_I+a_J=1$, $a_I,a_J\geq 0$.
\end{description}
\end{definition}

Thus, the greedy strategy is stationary in the investor's belief,
which can be computed by the advisor using the investor's initial belief $p_1$ and the messages sent to the investor in earlier rounds.

An important point is that  $\sigma_*$
does not depend on the discount factor, nor on the transition matrix. So it can be implemented without knowing $\delta$ nor $\pi$.

It will be convenient to identify, whenever there is no ambiguity,
a decomposition $p=a_I q_I+a_J q_J$ with the splitting $\mu$ which selects $q_I$ and $q_J$ with probabilities $a_I $ and $a_J$ respectively.
We will call the decomposition in \textbf{G2} \emph{the greedy splitting at $p$}.

As an illustration, consider Figure 1 below, with $\Omega=\{A,B,C\}$.
\begin{center}
\end{center}
\definecolor{xdxdff}{rgb}{0.49,0.49,1}
\definecolor{ffffff}{rgb}{1,1,1}
\definecolor{ttqqqq}{rgb}{0.2,0,0}
\definecolor{qqqqff}{rgb}{0,0,1}
\definecolor{uuuuuu}{rgb}{0.27,0.27,0.27}
\begin{tikzpicture}[line cap=round,line join=round,>=triangle 45,x=1.0cm,y=1.0cm]
\clip(-2.42,-3.04) rectangle (11.7,8.11);
\draw [color=ttqqqq] (0,0)-- (8.13,0);
\draw [color=ttqqqq] (8.13,0)-- (4.07,7.04);
\draw [color=ttqqqq] (4.07,7.04)-- (0,0);
\draw (1.64,2.84)-- (5.93,3.82);
\draw [dotted] (2.13,2.95)-- (3.74,0);
\draw [dotted] (3.52,3.27)-- (1.55,0);
\draw [dotted] (1.53,0.88)-- (5.13,4.04);
\draw (4.16,5.09) node[anchor=north west] {$I$};
\draw (5.43,1.97) node[anchor=north west] {$J$};
\draw (4.15,7.56) node[anchor=north west] {$A$};
\draw (-0.61,0.13) node[anchor=north west] {$B$};
\draw (8.27,0.23) node[anchor=north west] {$C$};
\draw (0.93,3.3) node[anchor=north west] {$B^+$};
\draw (6.05,4.3) node[anchor=north west] {$C^+$};
\draw (3.93,3.92) node[anchor=north west] {$\mathcal{F}$};
\draw (2.01,3.72) node[anchor=north west] {$q_I^{(1)}$};
\draw (3.23,3.97) node[anchor=north west] {$q^{(2)}_I$};
\draw (4.72,4.79) node[anchor=north west] {$q^{(3)}_I$};
\draw (3.64,0.01) node[anchor=north west] {$q^{(1)}_J$};
\draw (1.09,0.05) node[anchor=north west] {$q^{(2)}_J$};
\draw (0.96,1.38) node[anchor=north west] {$q^{(3)}_J$};
\draw (2.85,2.21) node[anchor=north west] {$p$};
\draw [color=ttqqqq] (1.64,2.84)-- (5.93,3.82);
\draw [color=ttqqqq] (5.93,3.82)-- (4.07,7.04);
\draw [color=ttqqqq] (4.07,7.04)-- (1.64,2.84);
\begin{scriptsize}
\draw [fill=uuuuuu] (0,0) circle (1.5pt);
\draw [fill=qqqqff] (8.13,0) circle (1.5pt);
\draw [fill=ffffff] (4.07,7.04) circle (1.5pt);
\draw [fill=xdxdff] (1.64,2.84) circle (1.0pt);
\draw [fill=xdxdff] (5.93,3.82) circle (1.0pt);
\draw [fill=xdxdff] (2.13,2.95) circle (1.0pt);
\draw [fill=xdxdff] (3.74,0) circle (1.0pt);
\draw [fill=xdxdff] (3.52,3.27) circle (1.0pt);
\draw [fill=xdxdff] (1.55,0) circle (1.0pt);
\draw [fill=uuuuuu] (2.7,1.91) circle (1.0pt);
\draw [fill=qqqqff] (5.13,4.04) circle (1.0pt);
\draw [fill=xdxdff] (1.53,0.88) circle (1.0pt);
\end{scriptsize}
\end{tikzpicture}

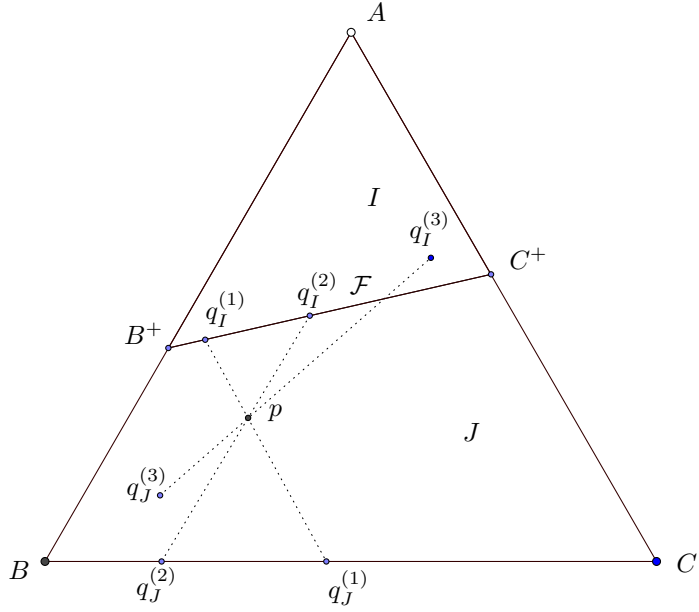
\centerline{Figure 1: Three splittings at $p$.}

\bigskip

Three different splittings at $p$ have been drawn:
$p=a_I^{(i)} q_I^{(i)}+a_J^{(i)} q_J^{(i)}$, $i\in \{1,2,3\}$,
with $\displaystyle a_I^{(i)}=\frac{\|p-q_J^{(i)}\|_2}{\|q_I^{(i)}-q_J^{(i)}\|_2}$,
so that $a_I^{(1)}>a_I^{(2)}>a_I^{(3)}$. Since $a_I^{(i)}$ is the current payoff under splitting $i$, the first of the three splittings yields a higher payoff.

For every two points $p_1,p_2 \in \Delta(\Omega)$ denote by $(p_1,p_2)$ the line that passes through $p_1$ and $p_2$,
and by $[p_1,p_2]$ the line segment that connects $p_1$ and $p_2$.
The noninvestment region $J$ is divided into two triangles by the segment $[B^+,C]$, see Figure 2 below.

\definecolor{xdxdff}{rgb}{0.49,0.49,1}
\definecolor{ffffff}{rgb}{1,1,1}
\definecolor{ttqqqq}{rgb}{0.2,0,0}
\definecolor{qqqqff}{rgb}{0,0,1}
\definecolor{uuuuuu}{rgb}{0.27,0.27,0.27}
\begin{tikzpicture}[line cap=round,line join=round,>=triangle 45,x=1.0cm,y=1.0cm]
\clip(-2.42,-3.01) rectangle (11.7,8.11);
\draw [color=ttqqqq] (0,0)-- (8.13,0);
\draw [color=ttqqqq] (8.13,0)-- (4.07,7.04);
\draw [color=ttqqqq] (4.07,7.04)-- (0,0);
\draw (1.64,2.84)-- (5.93,3.82);
\draw (4.16,5.09) node[anchor=north west] {$I$};
\draw (5.43,1.97) node[anchor=north west] {$J$};
\draw (4.15,7.56) node[anchor=north west] {$A$};
\draw (-0.61,0.13) node[anchor=north west] {$B$};
\draw (8.27,0.23) node[anchor=north west] {$C$};
\draw (0.93,3.3) node[anchor=north west] {$B^+$};
\draw (6.05,4.3) node[anchor=north west] {$C^+$};
\draw (3.69,4.04) node[anchor=north west] {$\mathcal{F}$};
\draw [color=ttqqqq] (1.64,2.84)-- (5.93,3.82);
\draw [color=ttqqqq] (5.93,3.82)-- (4.07,7.04);
\draw [color=ttqqqq] (4.07,7.04)-- (1.64,2.84);
\draw [dash pattern=on 2pt off 2pt] (1.64,2.84)-- (8.13,0);
\draw [dotted] (4.42,3.47)-- (8.13,0);
\draw [dotted] (1.64,2.84)-- (2.7,0);
\draw (2.46,1.56) node[anchor=north west] {$p$};
\draw (5.67,3.02) node[anchor=north west] {$p'$};
\begin{scriptsize}
\draw [fill=uuuuuu] (0,0) circle (1.5pt);
\draw [fill=qqqqff] (8.13,0) circle (1.5pt);
\draw [fill=ffffff] (4.07,7.04) circle (1.5pt);
\draw [fill=xdxdff] (1.64,2.84) circle (1.0pt);
\draw [fill=xdxdff] (5.93,3.82) circle (1.0pt);
\draw [fill=qqqqff] (2.28,1.13) circle (1.0pt);
\draw [fill=qqqqff] (5.53,2.44) circle (1.0pt);
\draw [fill=xdxdff] (4.42,3.47) circle (1.0pt);
\draw [fill=uuuuuu] (2.7,0) circle (1.0pt);
\end{scriptsize}
\end{tikzpicture}

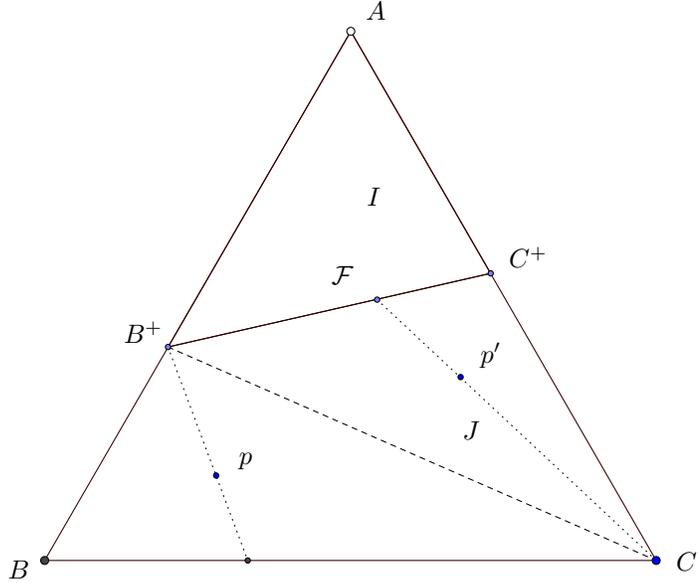
\centerline{Figure 2: The decomposition of the noninvestment region.}

\bigskip

Because the line $(B^+,C^+)$ has a positive slope,
every point $p$ in the lower triangle $(B^+,B,C)$ is split by the greedy strategy $\sigma_*$ between $B^+ $ and a point on the line segment $[B,C]$,
and points $p'$ in the upper triangle are split by $\sigma_*$ between $C$ and a point on the line segment $[B^+,C^+]$.

\bigskip
\noindent

Only in the case where the line $(B^+,C^+)$ is parallel to the line $(B,C)$ are there several optimal splittings.
This is a nongeneric situation\footnote{see lemma \ref{lemma:optimal solution} later.}
where $r(B) = r(C)$.

\begin{theorem}\label{th1}
If $|\Omega|=2$, the greedy strategy is optimal, irrespective of the initial distribution $p_1$.
\end{theorem}

Although interesting in its own sake, the two-state problem is specific in
many respects, and we next investigate the robustness of the
conclusion of Theorem \ref{th1}.

From now on, we restrict ourselves to a class of Markov chains,
in which shocks occur at random times, and the state remains unchanged between two consecutive shocks.
When a shock occurs, the next state is drawn according to a fixed distribution (and may thus coincide with the previous state).
The durations between successive shocks are i.i.d.~random variables with a geometric distribution.
Note that the invariant distribution $m$ is then equal to the fixed distribution according to which new states are drawn.
Equivalently, these are the chains with a transition function given by
\begin{eqnarray}
\pi(\omega \mid \omega)&=&(1-\lambda)m(\omega)+\lambda,\\
\pi(\omega' \mid \omega)&=&(1-\lambda) m(\omega')\mbox{ if }\omega'\neq \omega,
\end{eqnarray}
for some $\lambda\in [0,1)$. Note that the drift map $\phi:\Delta(\Omega)\to \Delta (\Omega)$
that describes the evolution of the investor's belief when no new information is provided
is given by
\[\phi(p)-m=\lambda(p-m),\]
so that $\phi$ is an homothety on the simplex with center $m$ and ratio $\lambda$. Notice that we only consider  homotheties with non negative ratio.

It turns out that even in this restricted class of chains, and with as few as three states,  Theorem \ref{th1} does not extend without qualifications.

\begin{proposition}\label{prop_counter}
Let $|\Omega|=3$. The greedy strategy need not be optimal for all initial distributions.
\end{proposition}

Indeed, we exhibit in Section \ref{sec_counter} a simple counterexample in which, for some initial distributions, it is strictly optimal not to disclose any information in early stages.

Yet, this counterexample hinges on fairly extreme choices of the invariant measure
and the initial distribution. In many cases the greedy strategy is a very relevant strategy. We substantiate this claim by means of three results.

First, it may be natural to assume that the initial distribution and the invariant measure coincide.\footnote{Or are very close. This is in particular relevant when the interaction between the advisor and the investor starts at a given date, long after the Markov chain has started evolving.}  In that case, the conclusion of Theorem \ref{th1} extends to an arbitrary number of states.

\begin{theorem}\label{th2}
Let the cardinality of $\Omega$ be arbitrary and suppose that $\phi$ is an homothety. If $p_1=m$, then the greedy strategy is optimal.
\end{theorem}

In fact, we will identify a polytope of initial distributions in $\Delta(\Omega)$
of full dimension that contains $m$ in its interior,
for which the greedy strategy is optimal. This allows us to prove that, irrespective of the initial distribution, it is eventually optimal to use the greedy strategy.

\begin{theorem}\label{th3}
Let the cardinality of $\Omega$ and the initial distribution be arbitrary, and suppose that $\phi$ is an homothety.
There is an optimal strategy $\sigma$ and an a.s.~finite stopping time after which $\sigma$ coincides with the greedy strategy.
\end{theorem}

Under the assumption that no two states yield the same payoff, the conclusion of Theorem \ref{th3} holds for \emph{every} optimal strategy $\sigma$.
That is, the suboptimality identified in Proposition \ref{prop_counter} is typically transitory. On almost every history, the advisor will at some point switch to the greedy strategy.
 Whether or not it is possible to put a \emph{deterministic} upper bound on this stopping time is  unknown to us.

We finally provide an in-depth analysis of the three-state case.
As it turns out, $\sigma_*$ is optimal in most circumstances.

When $|\Omega^-|=2$, we use the notations of Figure 1: $\Omega^-=\{B,C\}$ with
$r(B)\geq r(C)$, and the vertices of $\calF$ are denoted by $B^+$ and $C^+$.

\begin{theorem}\label{th4}
Assume $|\Omega|=3$  and suppose that $\phi$ is an homothety. The strategy $\sigma_*$ is optimal in the following cases:
\begin{itemize}
\item  $|\Omega^-|=1$;
\item $|\Omega^-|=2$ and $m$ belongs to either $I$ or to the triangle $(C^+,B^+,C)$.
\end{itemize}\end{theorem}

When instead $m$ belongs to the triangle $(B,B^+,C)$, the greedy strategy may fail to be optimal
only when three conditions are met simultaneously:
(i) the advisor
is very patient, that is, $\delta$ is close to one;
(ii) the state is very
persistent, that is, $\lambda$ is close to one; and
(iii) the line segment $\calF$ is
close to parallel to the line $(B,C)$, that is, $r(B)$ and $r(C)$ do not differ by
much. While the first two conditions are natural, we have
no intuition to offer for the last condition.

\section{Preparations}

\subsection{The greedy strategy}\label{sec-greedy}

In this section we provide more details and results on the greedy strategy $\sigma_*$. We let $\calE$ be the set of extreme points of $\calF$.
One can verify that for each $\omega^-\in \Omega^-$ and $\omega^+\in \Omega^+$,
the line segment $[\omega^-,\omega^+]$ contains a unique point in $\calE$ .
Conversely, any $e\in \calE$ lies on a line segment
$[\omega^-,\omega^+]$ for some  $\omega^-\in \Omega^-$, $\omega^+\in \Omega^+$.

It is convenient to reformulate the optimization problem \textbf{G2} in Definition \ref{def_greedy} as a linear program.
Given a finite set $A\subset \dR^\Omega$  we denote%
\footnote{Elements of  $\textrm{cone}(\Omega)$ are best seen as ``sub''-probability measures.} by $\mbox{cone}(A)$ the closed convex hull of $A\cup \{0\}$.
The optimization program in \textbf{G2} is equivalent to the following linear program
\[ (LP):\ \ \max \pi_1(\Omega), \]where the maximum is over
pairs $(\pi_1,\pi_2)\in \mbox{cone}(\calE)\times\mbox{cone}(\Omega^-)$ such that $\pi_1+\pi_2=p$.

\begin{lemma}
The value of the program (LP) is equal to the value of the following problem (LP').
\[ (LP'): \ \max \pi(\Omega),\] where the supremum is over all $\pi \in \mbox{cone}(\Omega)$
such that $\pi\leq p$ and $\displaystyle\sum_{\omega \in \Omega} \pi(\omega)r(\omega) \geq 0$.
\end{lemma}

\begin{proof} Recall that in \textbf{G2} $p\notin I$.
If $(\pi_1,\pi_2) \in \cone(\calE) \times \cone(\Omega^-)$ is an optimal solution of (LP)
then $\pi_1$ is a feasible solution of (LP'),
and therefore the value of (LP') is at least the value of (LP).

Fix now an optimal solution $\pi$ of (LP').
If $\sum_{\omega \in \Omega} \pi(\omega) r(\omega) > 0$,
then by increasing the weight of states in $\Omega^-$ we can increase $\pi(\Omega)$,
which would contradict the fact that $\pi$ is an optimal solution of (LP').
The weight of some states in $\Omega^-$ can be increased because $\pi \leq p$ and $\langle p,r\rangle < 0$.
It follows that $\pi \in \cone(\calE)$.
Set $\pi':=p-\pi\in \mbox{cone}(\Omega)$.
It is readily checked that $\pi'(\Omega^+)=0$, for otherwise the corresponding probability could be transferred to $\pi$. Hence $\pi'\in \mbox{cone}(\Omega^+)$
and $(\pi,\pi')$ is a feasible solution of (LP).
This implies that the value of (LP) is at least the value of (LP').
\end{proof}

\bigskip

This reformulation allows for a straightforward description of the greedy strategy at $p\in J$.
Intuitively, the weight $p(\omega^-)$ of each $\omega^-\in \Omega^-$ should be ``allocated'' between $\calF$ and $\Delta(\Omega^-)$
so as to maximize the total weight assigned to $\calF$. Since $\calF$ is defined by the equality $\langle \pi,r\rangle=0$,
it is optimal to allocate to $\calF$ the states $\omega^-$ in which the payoff $r(\omega^-)$ is the least negative.

Formally, we order the elements of $\Omega^-$ into $\omega_1,\ldots, \omega_{|\Omega^-|}$ by decreasing payoff:
$0>r(\omega_1)\geq \cdots \geq r(\omega_{|\Omega^-|})$.
Next, we define linear maps $L_1,\ldots, L_{|\Omega^-|}$ over $\Delta(\Omega)$ by
\[L_k(p):=\displaystyle \sum_{\omega\in \Omega^+}p(\omega)r(\omega) +\sum_{i\leq k} p(\omega_i)r(\omega_i).\]
$L_k(p)$ is a linear combination of the payoff of all states
whose payoff is positive or whose index is at most $k$;
that is, this linear combination assumes only states which are ``better'' than state $k$.

Observe that $L_1(\cdot)\geq \cdots \geq L_{|\Omega^-|}(\cdot)$.
We set $k^*:=\inf\{k: L_k(p)\leq 0\}$
to be the minimal index for which the linear combination $L_k(p)$ is nonpositive.
With these notations the optimal solution $\pi^*$ of $(LP')$ is given by

\begin{itemize}
\item $\pi^*(\omega^+)=p(\omega^+)$ for $\omega^+\in \Omega^+$;
\item $\pi^*(\omega_i)=p(\omega_i)$ for $i<k_*$;
\item $\pi^*(\omega_i)=0$  for $i>k_*$;
\item $\pi^*(\omega_{k^*})=-\frac{L_{k^*-1}(p)}{r(\omega_{k^*})}$.
\end{itemize}

The vector $\pi^*$ is the unique solution of (LP') as soon as
 no two states in $\Omega^-$ yield the same payoff.
If different states yield the same negative payoff, the ordering of $\Omega^-$ is nonunique.
 To sum up, we have proven the lemma below.

\begin{lemma}
\label{lemma:optimal solution}
Assume that no two states in $\Omega^-$ yield the same payoff: $r(\omega)\neq r(\omega')$ for
every $\omega\neq \omega'\in \Omega^-$. Then the greedy splitting $p=a_Iq_I+a_Jq_J$ is uniquely defined at each $p\in J$.
In addition, $q_I\in \calF$ and $q_J\in \Delta(\Omega^-)$.
\end{lemma}

The distributions $q_I$ and $q_J$ are obtained by renormalizing $\pi^*$ and $p-\pi^*$, respectively.
Note that for $p\in \Delta(\Omega^-)$ one has $a_I=0$, so that formally speaking, $q_I$ is indeterminate. Yet, the solution to (LP) is unique.

For $k\in \{1,\ldots, |\Omega^-|\}$, we let $\bar\calO(k):=\{p\in J\colon L_{k-1}(p)\geq 0\geq L_k(p)\}$ (with $L_0=1$).
The following figure depicts the sets $\bar\calO(1)$ and $\bar\calO(2)$ when $\Omega = \{\omega^+,\omega_1,\omega_2\}$,
$r(\omega^+) = 2$, $r(\omega_1) = -1$ and $r(\omega_2) = -4$.

\centerline{\includegraphics{figure.3}}
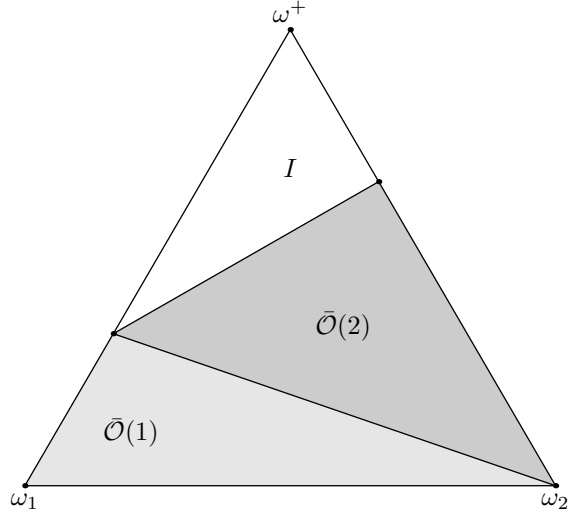
\centerline{Figure 3: The sets $\bar\calO(1)$ and $\bar\calO(2)$.}

A useful consequence of the solution of Problem (LP') is that
each set $\bar\calO(k)$ is stable under the greedy splitting.

\begin{lemma}
\label{lemma:stable}
If $p \in \bar\calO(k)$
and if $p = a_Iq_I+a_Jq_J$ is the greedy splitting at $p$,
then $q_I$ and $q_J$ are in $\bar\calO(k)$.
\end{lemma}

\begin{proof}
Fix $p \in \bar\calO(k)$.
Then the optimal solution $\pi^*$ to (LP') satisfies
\begin{eqnarray}
&&\pi^*(\omega^+) = p(\omega^+), \ \ \ \omega^+ \in \Omega^+,\\
&&\pi^*(\omega_i) = p(\omega_i), \ \ \ 1 \leq i \leq k-1,\\
&& 0 \leq \pi^*(\omega_k) \leq p(\omega_k).
\end{eqnarray}
By Lemma \ref{lemma:optimal solution},
$q_I$ is the normalization of $\pi^*$.
However, $L_{k-1}(\pi^*) > 0$ and $L_k(\pi^*) = 0$,
so that $L_{k-1}(q_I) > 0$ and $L_k(q_I) = 0$,
and therefore $q_I \in \bar\calO(k)$.

By Lemma \ref{lemma:optimal solution},
$q_J$ is the normalization of $p-\pi^*$.
This implies that $q_J(\omega^+) = 0$ for every $\omega^+ \in \Omega^+$
and $q_J(\omega_i) = 0$ for every $1 \leq i \leq k-1$,
so that $L_{k-1}(q_J) = 0$ and $L_k(q_J) \leq 0$,
and therefore $q_J \in \bar\calO(k)$.
\end{proof}

\subsection{Preparatory results}

For later use we collect in this section a number of simple,
yet general and useful observations. None of the results here uses the specific structure of the Markov chain.

For $p\in \Delta(\Omega)$ we let $\widehat r(p):=\max_{\mu\in \calS(p)}\mu(q\in I)$ be the highest stage payoff of the advisor when the investor's belief is $p$. Notice that $\widehat r$ coincides with the value function $V_0$ with null discount factor, so as an immediate corollary of lemma \ref{lemm1} we obtain:

\begin{lemma}\label{lemm-conc}
The map $\widehat r$ is concave.
\end{lemma}


\bigskip

Lemma \ref{lemm-conc} has the following noteworthy implication.
Fix $n \geq 1$ and let $\bar p_n:=\phi^{(n-1)}(p_1)$ be the (unconditional) distribution of the state in round $n$.
Then $\E_\sigma[p_n]=\bar p_n$ for every strategy $\sigma$ of the advisor.
In particular, by concavity of the function $\widehat r$ and Jensen's inequality,
the expected payoff of the advisor in round $n$ cannot exceed $\widehat r(\bar p_n)$, so that
\[\gamma_*(p_1):=(1-\delta) \sum_{n=1}^\infty \delta^{n-1} \widehat r(\bar p_n)\]
is an upper bound on the total discounted payoff to the advisor.
\bigskip

Fix $\delta<1$. We denote by $\gamma(p)$ the payoff induced by the greedy strategy as a function of the initial belief $p$. We also set
\[d(p):=\gamma(p)- \delta \gamma(\phi(p)).\]
For $p\in J$, the quantity $d(p)$ is the payoff difference when playing greedy, compared to disclosing
\emph{no} information in the first round and  then switching to the greedy strategy in round 2.

If the greedy strategy is optimal for \emph{all} initial distributions, then $\gamma$ coincides with $V_\delta$,
and therefore $\gamma$ is concave
and $d(\cdot)\geq 0$ over $\Delta(\Omega)$. Somewhat surprisingly, the converse implication also holds.

\begin{lemma}\label{lemmDPP}
Assume that $\gamma$ is concave and that $d \geq 0$ over $J$. Then $\sigma_*$ is optimal for all initial distributions.
\end{lemma}

\begin{proof}
It suffices to show that $\gamma(\cdot)$ solves the dynamic programming equation, that is,
\[\gamma(p)=\max_{\mu\in \calS(p)}\left\{ (1-\delta) \mu(q\in I)+\delta \E _\mu\left[ (\gamma\circ\phi)(q)\right]\right\}.\]
Denoting by $\mu^*_p$ the greedy splitting at $p$, we have
\[ \gamma(p) = (1-\delta)\mu^*_p(q \in I) + \delta \E_{\mu^*_p}[\gamma\circ\phi(p)] \]
and therefore
\[\gamma(p)\leq\max_{\mu\in \calS(p)}\left\{ (1-\delta) \mu(q\in I)+\delta \E _\mu\left[ (\gamma\circ\phi)(q)\right]\right\}.\]
We now show the reverse inequality.
Let $\mu\in \calS(p)$ be arbitrary.
Because $d(\cdot) \geq 0$ on $J$, one has for each $q\in \Delta(\Omega)$,
\[(1-\delta)1_{\{q\in I\}} +\delta (\gamma\circ\phi(q))\leq \gamma(q).\]
Taking expectations w.r.t.~$\mu$ and using the concavity of $\gamma$, one gets
\[(1-\delta) \mu(q\in I)+\delta \E_\mu\left[ (\gamma\circ \phi)(q)\right] \leq \E_\mu\left[ \gamma(q)\right]\leq \gamma(\E_\mu[q]) = \gamma(p).\]
 This concludes the proof.
\end{proof}

\section{The 2-state case: proof of Theorem \ref{th1}}\label{sec-twostates}

We here assume that $\Omega=\{\omega^-,\omega^+\}$ is a two-point set.
W.l.o.g. we assume that $r(\omega^+)>0>r(\omega^-)$,
and we identify a belief over $\Omega$ with the probability assigned to state $\omega^+$.
Here, the investor is willing to invest as soon as the probability assigned to $\omega^+$ is high enough,
and the investment region is the interval $I=[p_*,1]$, where $p_*\in (0,1)$ solves $p_* r(\omega^+) +(1-p_*)r(\omega^-)=0$.

At any $p<p_*$, $\sigma_*$ chooses the distribution $\mu\in \calS(p)$ which assigns probabilities $\displaystyle \frac{p}{p_*}$
and $\displaystyle 1-\frac{p}{p_*}$ to $p_*$ and $ 0$, respectively, and does not disclose information if $p\geq p_*$. In particular,
\begin{equation}\label{eq1}\gamma(p)=\frac{p}{p_*} \gamma(p_*) +\left(1-\frac{p}{p_*} \right)\gamma(0)\mbox{ for } p\in [0,p_*],\end{equation}
and
\[\gamma(p)=(1-\delta)+\delta (\gamma\circ \phi)(p) \mbox{ for } p\in [p_*,1].\]
Eq.~(\ref{eq1}) shows that $\gamma(\cdot)$ is affine over $[0,p_*]$ (but need not be affine on $[p_*,1]$). Note that
$\gamma(0)=\delta (\gamma\circ \phi)(0)$.

In this setup,%
\footnote{This observation does not generalize to $|\Omega|\geq 3$.}
concavity of $\gamma(\cdot)$ alone is equivalent to the optimality of $\sigma_*$.
Indeed, assume $\gamma(\cdot)$ is concave.
Recall that $d(p)=0$ for $p\in I$.
In addition, $\gamma(\cdot)$ is affine on $[0,p_*]$,
and since $\phi$ is affine and $\gamma$ is concave, the composition $(\gamma\circ \phi)(\cdot)$ is concave on $[0,p_*]$.
Thus, $d(\cdot)$ is convex on $[0,p_*]$. Observe now that $d(0)=d(p_*)=0$,
hence $d(\cdot)\geq 0$ on $\Delta(\Omega)$, and the optimality of $\sigma_*$ then follows from Lemma  \ref{lemmDPP}.

\bigskip

It is left to prove that $\gamma$ is concave.
The invariant measure $m$ assigns probability $\displaystyle \frac{\pi(\omega^+\mid \omega^-)}{\pi(\omega^+\mid \omega^-)+\pi(\omega^-\mid \omega^+)}$
to $\omega^+$. With our notations, for $q\in [0,1](=\Delta(\Omega)$) one has
\[\phi(q)=m+(1-\pi(\omega^+\mid \omega^-)- \pi(\omega^-\mid \omega^+))(q-m),\]
hence $\phi$ is a homothety on $[0,1]$ centered at $m$ with ratio $\lambda :=1-\pi(\omega^+\mid \omega^-)- \pi(\omega^-\mid \omega^+)\in (-1,1)$.

It is convenient to organize the proof below according to the relative values of $p_*$ and $m$, and to the sign of the ratio $\lambda$.
In the first case we provide a direct argument.
In the following cases we prove the concavity of $\gamma$.


\noindent{\bf Case 1: } $p^*\geq m$ and  $\lambda \geq 0$.

 \unitlength 0,7mm
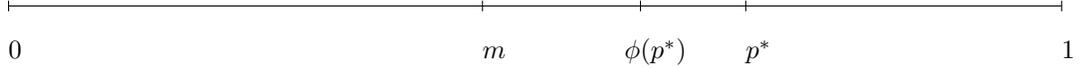
\begin{picture}(200,20)
\put(0,10){\line(1,0){200}}
 \put(0,0){$0$}
  \put(0,9){\line(0,1){2}}
   \put(200,0){$1$}
  \put(200,9){\line(0,1){2}}
   \put(90,0){$m$}
  \put(90,9){\line(0,1){2}}
     \put(140,0){$p^*$}
  \put(140,9){\line(0,1){2}}
       \put(117,0){$\phi(p^*)$}
  \put(120,9){\line(0,1){2}}
\end{picture}

Establishing directly the concavity of $\gamma(\cdot)$ is
possible, yet involved, as $\gamma(\cdot)$ fails to be affine on
$I$. We instead argue that $\gamma(p)=\gamma_*(p)$ for each $p$,
where $\gamma_*(p)$ is the upper bound on payoffs identified
earlier.

Assume first that $p_1\in [0,p_*]$.
Since the interval $[0,p_*]$ is stable under $\phi$
under $\sigma_*$, one has $q_n\in
\{0,p_*\}$ for each $n\geq 1$, and $p_n\in \{\phi(0),
\phi(p_*)\}$ for each $n>1$.
In each stage $n\geq 1$,
conditional on the previous history, the strategy $\sigma_*$
maximizes the expected payoff in stage $n$, so that the expected
payoff in stage $n$ is given by $\E_{\sigma_*}\left[ \widehat
r(p_n) \right]$. Since $\widehat r$ is affine on $[0,p_*]$, the expected payoff in stage $n$ is
also equal to $\widehat r\left( \E_{\sigma_*}\left[ p_n
\right]\right)=\widehat r (\bar p_n)$, so that
$\gamma(p_1)=\gamma_*(p_1)$.

Assume now that $p_1\in I$. Then the sequence $(\bar p_n)_{n\geq
1}$ is decreasing (towards $m$). Let $n_*:=\inf\{ n\geq1: \hat
p_n<p_*\}$ be the  stage in which the unconditional distribution
of the state leaves $I$. Under $\sigma_*$, the advisor discloses
no information up to stage $n_*$, so that $r(q_n)=1=\widehat
r(\bar p_n)$ for all $n<n_*$. That is, $\sigma_*$ achieves
the upper bound on  the payoff in \emph{each} stage $n<n_*$, and,
by the previous argument, in each stage $n\geq n_*$ as well.

\vspace{1cm}

\noindent{\bf Case 2: }  $p^*\leq m$ and  $\lambda\geq 0$.

 \unitlength 0,7mm
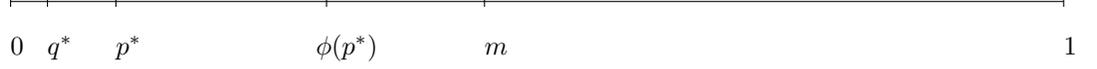
\begin{picture}(200,20)
\put(0,10){\line(1,0){200}}
 \put(0,0){$0$}
  \put(0,9){\line(0,1){2}}
   \put(200,0){$1$}
  \put(200,9){\line(0,1){2}}
   \put(90,0){$m$}
  \put(90,9){\line(0,1){2}}
     \put(20,0){$p^*$}
  \put(20,9){\line(0,1){2}}
       \put(58,0){$\phi(p^*)$}
  \put(60,9){\line(0,1){2}}

         \put(7,0){$q^*$}
  \put(7,9){\line(0,1){2}}
\end{picture}

Since $m\geq p_*$, one has $\phi([p_*,1])\subseteq [p_*,1]$: the investment region is stable under $\phi$.
Thus, once in $I$, $\sigma_*$ yields a payoff of 1 in each stage: $\gamma(p)=1$ for $p\geq p_*$. Using (\ref{eq1}), one thus has
\[\gamma(p)=\frac{p}{p_*}+\left(1-\frac{p}{p_*}\right) \gamma(0)\mbox{ for }p>p_*.\]
Since $\gamma(0)<1$ it follows that $\gamma$ is increasing (and affine) on $[0,p_*]$. Hence the concavity of $\gamma$ on $[0,1]$.

\vspace{1cm}

\noindent{\bf Case 3: }  $p^*\geq m$ and  $\lambda\leq 0$.

 \unitlength 0,7mm
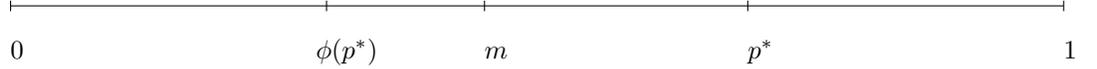
\begin{picture}(200,20)
\put(0,10){\line(1,0){200}}
 \put(0,0){$0$}
  \put(0,9){\line(0,1){2}}
   \put(200,0){$1$}
  \put(200,9){\line(0,1){2}}
   \put(90,0){$m$}
  \put(90,9){\line(0,1){2}}
     \put(140,0){$p^*$}
  \put(140,9){\line(0,1){2}}
       \put(58,0){$\phi(p^*)$}
  \put(60,9){\line(0,1){2}}
\end{picture}

Recall that $\gamma$ is affine on $[0,p_*]$. From the formula $\phi(p)=m+\lambda(p-m)$,
one has $\phi(p)\leq \phi(p_*)\leq p_*$ for all $p\geq p_*$, that is, $I$ is mapped into $[0,p_*]$ under $\phi$.
Since
\begin{equation}\label{eq2}
\gamma(p)=(1-\delta)+\delta (\gamma\circ\phi)(p)\mbox{ for }p\in I,
\end{equation}
this implies that $\gamma$ is also affine on $[p_*,1]$. To establish the concavity of $\gamma$
we need to compare the slopes of $\gamma$ on $I$ and $J = [0,p^*)$.
Differentiating (\ref{eq2}) yields $\gamma'(p)=\delta \lambda (\gamma'\circ \phi)(p)$ for $p>p_*$,
hence the two slopes are of opposite signs. Note finally that $\gamma(p_*)=(1-\delta)+\delta (\gamma\circ \phi)(p_*)$,
hence $\gamma(p_*)>(\gamma\circ \phi)(p_*)$, so that $\gamma$ is increasing on $[0,p_*]$ (and then decreasing on $[p_*,1]$).

\vspace{1cm}

\noindent{\bf Case 4: } $p^*\leq m$ and  $\lambda\leq 0$.

 \unitlength 0,7mm
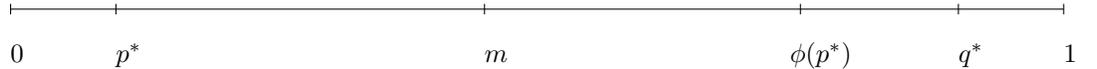
\begin{picture}(200,20)
\put(0,10){\line(1,0){200}}
 \put(0,0){$0$}
  \put(0,9){\line(0,1){2}}
   \put(200,0){$1$}
  \put(200,9){\line(0,1){2}}
   \put(90,0){$m$}
  \put(90,9){\line(0,1){2}}
     \put(20,0){$p^*$}
  \put(20,9){\line(0,1){2}}
       \put(180,0){$q^*$}
  \put(180,9){\line(0,1){2}}

       \put(148,0){$\phi(p^*)$}
  \put(150,9){\line(0,1){2}}
\end{picture}

The dynamics of the belief under $\sigma_*$ is here slightly more
complex. If $\phi(1)\geq p_*$, the investment region $I$ is stable
under $\phi$, hence $\gamma(p)=1$ for all $p\in I$, and the
concavity of $\gamma$ follows as in Case 2. If instead
$\phi(1)<p_*$, we introduce the cutoff $q_*\in [m,1]$ defined by
$\phi(q_*)=p_*$. Since $\phi$ is contracting, the length of the
interval $[\phi(q_*),\phi(p_*)]$ is smaller than that of
$[p_*,q_*]$, which implies that the interval $[p_*,q_*]$ is stable
under $\phi$. Therefore, $\gamma(p)=1$ for all $p\in [p_*,q_*]$.
As in Case 2,  this implies that $\gamma$ is increasing (and
affine) on $[0,p_*]$.

For $p\geq q_*$, $\gamma(p)=(1-\delta)+(\gamma\circ\phi)(p)$. Since
the interval $[q_*,1]$ is mapped into $[0,p_*]$ under $\phi$, this
implies in turn that $\gamma$ is affine on $[q_*,1]$, with slope
given by $\gamma'(p)=\lambda \delta (\gamma'\circ \phi)(p)<0$. That
is, $\gamma$ is piecewise affine, increasing on $[0,p_*]$,
constant on $[p_*,q_*]$ and decreasing on $[q_*,1]$.

\section{A counterexample: proof of Proposition \ref{prop_counter}}\label{sec_counter}

We here provide an example in which $\sigma_*$ fails to be optimal for some initial distribution $p_1$.

There are three states, $\Omega = \{\omega_1,\omega_2,\omega_3\}$,
and the investment region is the triangle with vertices $\omega_1$, $\ep \omega_1+(1-\ep)\omega_2$,
and $\frac{1}{2}\omega_1+\frac{1}{2}\omega_3$ where $\ep > 0$ is sufficiently small (see Figure 1).
Assume first that the invariant distribution is $m=\omega_2$, and that $\lambda=\frac{1}{2}$. Let the initial belief be $p_1 = 2\ep\omega_1+(1-2\ep)\omega_3$.

\centerline{\includegraphics{figure.1}}

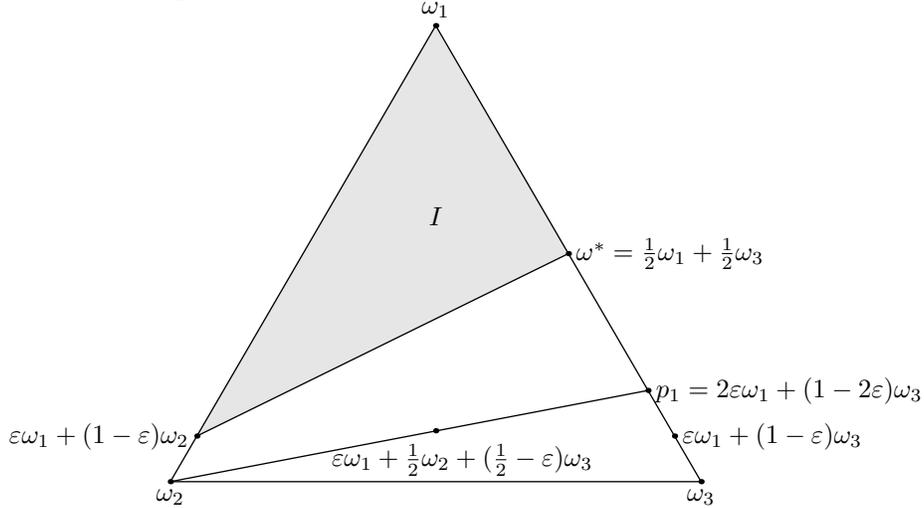
\centerline{Figure 4: The counterexample.}

\bigskip

According to $\sigma_*$, at the first stage $p_1$ is split between $\omega_3$ (with probability $1-4\ep$) and $\frac{1}{2}\omega_1+\frac{1}{2}\omega_3$
(with probability $4\ep$).
Because the line segment $[\omega_2,\omega_3]$ is contained in $J$ and $m=\omega_2$, the payoff to the investor once the belief reaches $\omega_3$ is 0.
It follows that the payoff under $\sigma_*$ is $\gamma(p_1) = 4\ep\gamma(\frac{1}{2}\omega_1+\frac{1}{2}\omega_3) \leq 4\ep$.

Consider the alternative strategy, in which the advisor discloses no information in the first stage, so that
\[p_2=\phi(p_1)=\tfrac{1}{2}\omega_2+\tfrac{1}{2}p_1 = \ep\omega_1+\tfrac{1}{2}\omega_2+(\tfrac{1}{2}-\ep)\omega_3,\]
and then at the second stage splits $p_2$ between $q_2=\ep\omega_1+(1-\ep)\omega_2$ (with probability $\frac{1}{2(1-\ep)}$)
and $q_2=\ep\omega_1+(1-\ep)\omega_3$ (with the complementary probability).
The expected  payoff in the second stage is therefore $\displaystyle \tfrac{1}{2(1-\ep)}$.

Hence, the alternative strategy is better than $\sigma_*$ as soon as $\displaystyle 4\ep < \tfrac{1}{2(1-\ep)}\times \delta(1-\delta)$.
For fixed $\delta\in (0,1)$, this is the case for small $\ep$.

In this example, the invariant distribution $m$ is on the boundary of $\Delta(\Omega)$. However, for fixed $\delta$, the above argument
is robust to a perturbation of the transition probabilities.
In particular we obtain a similar result for an invariant distribution $m$ that is in the interior of $\Delta(\Omega)$.

The example shows that the greedy strategy is not always optimal. Then,
a natural question is   whether there is always an optimal strategy that satisfies the following property:
whenever the strategy provides information to the investor, it does so according to the greedy splitting.
The answer is negative, and we end this section by showing that the optimal strategy in this
example sometimes splits the investor's belief in a way that is
not the greedy splitting.

Assume then to the contrary that in this example there is an optimal strategy $\sigma$ that,
at every belief $p$, either does not provide information or reveals information according to the greedy splitting.
Consider the line segment
$[\omega^*,\omega^3]$. If there is a belief $p$ on this line
segment for which the greedy splitting is optimal, then by Lemma
\ref{lemm18} below the greedy splitting is optimal for every belief on
this line segment, which contradicts the fact that the greedy
splitting is not optimal at
$2\varepsilon\omega_1+(1-2\varepsilon)\omega_3$.
Thus, $\sigma$ does not provide information for any $p$ on this line segment.
In particular,
\[ V_\delta(p) = (1-\delta)\textbf{1}_{\{p \in \calF\}} + \delta V_\delta(\phi(p)). \]
Whereas the functions $V_\delta$ and $V_\delta \circ \phi$ are continuous,
the function $\textbf{1}_{\{p \in \calF\}}$ is not continuous on the line segment $[\omega^*,\omega^3]$, a contradiction.

\section{Invariant initial distributions: proof of Theorem \ref{th2}}

We will prove a strengthened version of Theorem \ref{th2}, which will be used in the proof of Theorem \ref{th3}.
We recall from Section \ref{sec-greedy} that  $L_k(\cdot)$ is the linear map defined by
\[L_k(p)=\sum_{\omega\in \Omega^+}p(\omega)r(\omega) +\sum_{i\leq k} r(\omega_i)p(\omega_i).\]
Note that, with the notations of Section \ref{sec-greedy}, the map $p\mapsto \pi^*_1(p)$ is affine on the set
$\bar\calO(k)=\{L_k(\cdot)\geq 0\geq L_{k+1}(\cdot)\}$ and, therefore, so is $\widehat r(\cdot)$.

\begin{theorem}\label{th2bis}
Let $k$ be such that
$m\in \bar \calO(k)$. Then $\gamma(p_1)=\gamma_*(p_1)$ for every initial distribution $p_1\in \bar\calO(k)$.
In particular, the greedy strategy $\sigma_*$ is optimal whenever $p_1\in \bar\calO(k)$.
\end{theorem}

\bigskip

\begin{proof}
Let $p\in \bar\calO(k)$ be arbitrary, and denote by $p=a_Iq_I+a_Jq_J$ the greedy splitting at $p$.
Lemma \ref{lemma:stable} implies that both $q_I$ and $q_J$ belong to $\bar\calO(k)$.
Since $m\in \bar\calO(k)$, the set $\bar\calO(k)$ is stable under the greedy strategy $\sigma_*$.
That is, if $p_1 \in \bar\calO(k)$ then under $\sigma^*$ we have $p_n \in \bar\calO(k)$ for every $n$.

Since $\widehat r$ is affine on $\bar\calO(k)$, one has for each stage $n$,
\[\E_{\sigma_*}\left[ \widehat r(p_n)\right]=\widehat r\left( \E_{\sigma_*}\left[ p_n\right]\right)= \widehat r(\bar p_n).\]
Hence the result.
\end{proof}

\section{Eventually greedy strategies: proof of Theorem \ref{th3}}

We will assume that $m\in J$, which is the more difficult case. The case where $m\in I$ is dealt with at the end of the proof.
We start with an additional, simple, observation on the shape of the value function.

\begin{lemma}\label{lemm18}
Let $p\in J$ be given, and let $p=a_Iq_I+a_Jq_J$ be an \emph{optimal} splitting at $p$. If $a_I,a_J>0$, then
\begin{enumerate}
\item $V_\delta$ is affine on $[q_I,q_J]$;
\item at each $p'\in [q_I,q_J], $ it is optimal to split between $q_I$ and $q_J$.
\end{enumerate}
\end{lemma}

We stress that $p=a_Iq_I+a_Jq_J$ need not be the greedy splitting at $p$.

\begin{proof}
By assumption, $V_\delta(p)=a_IV_\delta(q_I)+a_JV_\delta(q_J)$, hence the first statement follows from the concavity of $V_\delta$ on $[q_I,q_J]$. Given a point $p'=a'_Iq_I+a'_Jq_J\in [q_I,q_J]$,
this affine property implies
\[V_\delta(p')=a'_IV_\delta(q_I)+a'_JV_\delta(q_J).\]
On the other hand, splitting $p'$ into $q_I$ and $q_J$ yields $a'_IV_\delta(q_I)+a'_JV_\delta(q_J)$, hence the second
statement. \end{proof}

\bigskip

In the sequel, we let $k$ be such that $m\in \bar\calO(k)$. By Theorem \ref{th2bis}, $\gamma(p)=\gamma_*(p)$ for every $p\in \bar\calO(k)$. We denote by $\bar J:=J\cup \calF=\{p\in \Delta(\Omega), \langle p,r\rangle \leq 0\}$ the closure of $J$.

\begin{lemma}\label{lemm19}
Let $p\in \bar J\setminus \bar\calO(k)$ be given, and let $p=a_Iq_I+a_Jq_J$ be an optimal splitting at $p$.
Then $[q_I,q_J]\cap \bar\calO(k)=\emptyset$.
\end{lemma}

\begin{proof}
We argue by contradiction and assume that there exists $p'\in \bar\calO(k)\cap [q_I,q_J]$.
By Lemma \ref{lemm18}, the splitting $p'=a'_Iq_I+a'_Jq_J$ is optimal at $p'$.
Since $p'\in \bar\calO(k)$, one has $V_\delta(p')=\gamma_*(p')$.
This implies that under the optimal strategy,
the expected payoff in each stage is equal to the first best payoff in that stage.
In particular, any optimal splitting at $p'$ must be the greedy one.
By Lemma \ref{lemma:stable} this implies that both $q_I$ and $q_J$ belong to $\bar\calO(k)$,
hence by convexity $p \in \bar\calO(k)$ -- a contradiction.
\end{proof}

\bigskip

We will need to make use of a set $P$ of the same type as $\bar\calO(k)$, which contains $m$ in its interior, and starting from which $\sigma_*$ is optimal.

If $m$ belongs to the interior of $\bar\calO(k)$ for some $k$, we simply set $P:=\bar\calO(k)$. Otherwise, one has
\begin{equation}\label{eq21}L_{k-1}(m)>0=L_k(m)=\cdots=L_l(m)>L_{l+1}(m)\mbox{ for some }k\leq l.\end{equation}
We then set $P:=\{p\in \bar J, L_{k-1}(p)\geq 0\geq L_{l+1}(p)\}=\bar\calO(k-1)\cup\cdots \cup \bar\calO(l)$.
By construction, $m$ belongs to the interior of $P$.
By (\ref{eq21}), one has $m\in \bar\calO(i)$ for $i=k-1,\ldots, l$,
hence the set $P$ is stable under the Markov chain.
This implies that $\sigma_*$ is optimal whenever $p_1\in P$.

\begin{lemma}\label{lemm22}
Assume that all connected components of $\bar J\setminus P$ in $\Delta(\Omega)$ are convex. Then the conclusion of Theorem \ref{th3} holds.
\end{lemma}

\begin{proof}
Let $\calC$ be an arbitrary connected component of $\bar J\setminus P$. Since $\calC$ is convex, there is an hyperplane $H$ (in $\Delta(\Omega)$)
that (weakly) separates $\calC$ from $P$, and we denote by $Q$ the open half-space of $\Delta(\Omega)$ that contains $m$.

We will make use of the following observation. Since $\bar Q \cap \Delta(\Omega)$ is compact, there is a constant $c>0$ such that the following holds: for all $\tilde p\in \bar Q\cap \Delta(\Omega)$ and all $\mu\in \calS(\tilde p)$, one has $\mu(q\in \bar Q)\geq c$.

Since $m\in Q$, the distance from $m$ to $H$ is positive.
Since $\Delta(\Omega)$ is compact and $\phi$ is contracting, there exists $\bar n\in \dN$ such that $\phi^{(\bar n)}(p)\in Q$ for all $p\in \Delta(\Omega)$.

Fix $p\in \calC$ and let $\tau$ be any optimal policy when $p_1=p$.
We let $\theta:=\inf\{n\geq 1,q_n\in P\}$ be the stage at which the investor's belief reaches $P$.
We prove below that $\theta <+\infty$ with probability 1 under $\tau$.
This proves the result, since $\theta$ is an upper bound on the actual stage at which the advisor can switch to $\sigma_*$.

Since $m\in J$, under $\tau$ one has $q_n\in \bar J$ with probability 1 for all $n$.
By Lemma \ref{lemm19}, one has $q_n\in\calC$ on the event $n<\theta$. On the other hand, the (unconditional) law of $q_n$ belongs to $\calS(\bar p_n)$ for each $n$: $\E\left[ q_n\right]=\bar p_n$. This implies that $\prob_\tau(q_{\bar n}\in Q)\geq c$, so that $\prob_\tau(\theta\leq \bar n)\geq c$.

The same argument, applied more generally, yields $\prob_\tau(\theta\leq (j+1)\bar n\mid \theta > j\bar n)\geq c$ for all $j\in \dN$.
Therefore, $\prob(\theta <+\infty)=1$, as desired.
\end{proof}

\bigskip

The complement of $P$ in $\bar J$ is the \emph{disjoint} union of $\{p\in J: L_k(p)<0\}$ and $\{p\in J: L_l(p)>0\}$. Both sets are convex, hence Theorem \ref{th3} follows from Lemma \ref{lemm22}.

\bigskip

For completeness, we now provide a proof for the case $m\in I$.
In that case, the entire investment region $I$ is stable under $\sigma_*$.
Hence, it is enough to prove that the stopping time  $\theta:=\inf\{n\geq 1 \colon q_n\in I\}$ is a.s.~finite,
for any initial distribution $p\in J$ and any optimal policy $\tau$.
Observe first that the payoff $\gamma(p)$ under $\sigma_*$ is bounded away from zero and therefore so is $V_\delta(p)\geq \gamma(p)$.
For a fixed $\delta$, this implies the existence of a constant $c>0$ and of  a stage $\bar n\in \dN$, such that $\prob_\tau(\theta \leq \bar n)\geq c$. This implies the result, as in the first part of the proof.

\section{The case of 3 states: proof of Theorem \ref{th4}}

The analysis relies on a detailed study of the belief dynamics under $\sigma_*$.
We will exhibit a simplicial decomposition of $\Delta(\Omega)$ with respect to which $\gamma$ is affine.
This partition will be used to prove that $\gamma(\cdot)$ is concave and $d(\cdot)$ nonnegative on $\Delta(\Omega)$.
We will organize the discussion in two cases, depending on the size of $\Omega^-$.

\bigskip\noindent

\textbf{Case 1}: $\Omega^-=\{C\}$.

We prove the optimality of $\sigma_*$ in two steps. We first argue that $\gamma$ is concave and $d$ nonnegative on the straight line joining $C$ and $m$. We next check that both $\gamma$ and $d$ are constant on each line parallel to $\calF$. These two steps together readily imply that $\gamma$ is concave and $d$ nonnegative throughout $\Delta(\Omega)$, as desired.
\bigskip

\underline{Step 1}. Denote by $\calL$ the line $(C,m)$, and by $p_*$ the intersection of $\calL$ and $\calF$.
The line $\calL$ is stable under $\phi$, and $\sigma_*$ splits any $p\in \calL\cap J$ between $C$ and $p_*$.
The dynamics of beliefs and of payoffs thus follows the same pattern as in the two-state case.
Hence%
\footnote{We emphasize however that this is not sufficient to conclude the optimality of $\sigma_*$ on $\calL$.}
it follows from Section \ref{sec-twostates} that $\gamma$ is concave and $d(\cdot)$ nonnegative on $\calL$.

\definecolor{xdxdff}{rgb}{0.49,0.49,1}
\definecolor{ttqqqq}{rgb}{0.2,0,0}
\definecolor{qqqqff}{rgb}{0,0,1}
\definecolor{uuuuuu}{rgb}{0.27,0.27,0.27}
\begin{tikzpicture}[line cap=round,line join=round,>=triangle 45,x=1.0cm,y=1.0cm]
\clip(-2.42,-2.56) rectangle (16.15,8.11);
\draw [color=ttqqqq] (0,0)-- (8.13,0);
\draw [color=ttqqqq] (8.13,0)-- (4.07,7.04);
\draw [color=ttqqqq] (4.07,7.04)-- (0,0);
\draw (1.95,2.1) node[anchor=north west] {$I$};
\draw (3.57,0.87) node[anchor=north west] {$J$};
\draw (8.27,0.23) node[anchor=north west] {$C$};
\draw (3.9,2.56) node[anchor=north west] {$\mathcal{F}$};
\draw (1.87,0)-- (6.21,3.33);
\draw [dash pattern=on 2pt off 2pt] (4.09,0)-- (6.89,2.15);
\draw [dotted] (3.56,1.29)-- (8.13,0);
\draw [dotted] (4.44,4.08)-- (8.13,0);
\draw (4,4.69) node[anchor=north west] {$m$};
\draw (5.77,2.23) node[anchor=north west] {$p_{\mathcal{L}}$};
\draw (5.1,0.8) node[anchor=north west] {$p$};
\draw (5.46,3.59) node[anchor=north west] {$p_*$};
\begin{scriptsize}
\draw [fill=uuuuuu] (0,0) circle (1.5pt);
\draw [fill=qqqqff] (8.13,0) circle (1.5pt);
\draw [fill=uuuuuu] (4.07,7.04) circle (1.5pt);
\draw [fill=xdxdff] (1.87,0) circle (1.0pt);
\draw [fill=xdxdff] (6.21,3.33) circle (1.0pt);
\draw [fill=qqqqff] (5.18,0.83) circle (1.0pt);
\draw [fill=qqqqff] (4.44,4.08) circle (1.0pt);
\draw [fill=uuuuuu] (5.57,2.84) circle (1.0pt);
\draw [fill=uuuuuu] (6.48,1.83) circle (1.0pt);
\draw [fill=uuuuuu] (3.56,1.29) circle (1.0pt);
\end{scriptsize}
\end{tikzpicture}
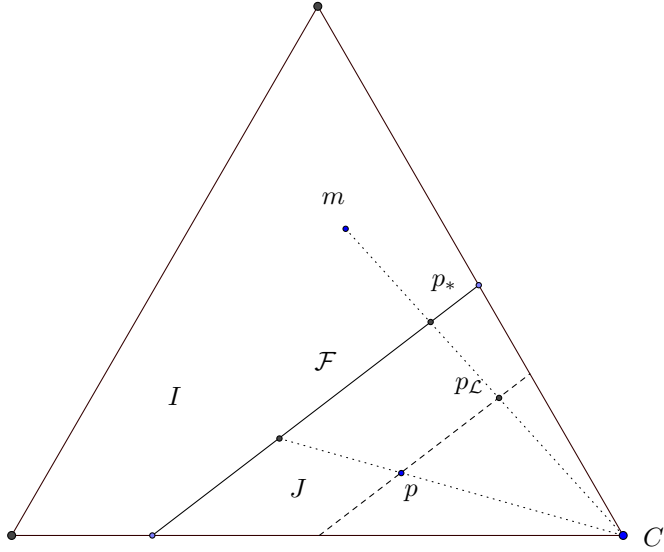
\centerline{Figure 5: The case $|\Omega^-|=1$.}

\bigskip

\underline{Step 2}.
With the notations of Figure 5, $\sigma_*$
splits any $p\in J$ between $C$ and a point in the investment frontier $\calF$, and $\widehat
r(p)=\widehat r(p_\calL)$, where $p_\calL$ is a point for which $(pp_\calL)$ is parallel to
$\calF$.
Note that any line parallel to $\calF$ is mapped by $\phi$
into some line parallel to $\calF$. This implies that $\gamma$ and
$\gamma\circ \phi$ are constant on each line parallel to $\calF$,
and so is $\Delta(\cdot)$.

\bigskip\noindent

Denote by $J_0$ the triangle $(C^+,B^+,C)$.

\textbf{Case 2}: $\Omega^-=\{B,C\}$ and $m\in I\cup J_0$.

 Again, we proceed in several steps. We first prove that $\gamma$ is concave and $d$ nonnegative on $I \cup J_0$. We next explicit the dynamics of beliefs under $\sigma_*$. This in turn leads to the concavity of $\gamma$ in Step 3. In Step 4, we prove that $d\geq 0$ on $\Delta(\Omega)$.

 \bigskip

 \underline{Step 1}. The function $\gamma$ is concave and $d\geq 0$ on $I\cup J_0$.

 The analysis is identical to that in \textbf{Case 1}. First, it follows from the two-state case that the conclusion holds on the line $(C,m)$.
 Next, as before, both $\gamma$ and $\gamma\circ \phi$ are constant on each line segment contained in $I\cup J_0$ and parallel to $\calF$.

 \bigskip

 \underline{Step 2}. The belief dynamics under $\sigma_*$.

 We construct recursively a finite sequence $O_1,\ldots, O_K$ of points in the line segment $[B,C]$ as follows.
 Set first $O_1=C$ and let $k\geq 1$. If $\phi$ maps $B$ into the triangle $(C^+,O_k,O_{k-1})$ (or $J_0$, if $k=1$), we set $K=k$.
 Otherwise, $O_{k+1}$ is the unique point in the line segment $[C,O_k]$ such that $P_{k+1}:=\phi(O_{k+1})\in [C,O_k]$.

 Since $\phi$ is an homothety, all points $(P_k)_{k\leq K}$ lie on some line $\calP$ parallel to $(B,C)$, see Figure 6.

\definecolor{xdxdff}{rgb}{0.49,0.49,1}
\definecolor{ffffff}{rgb}{1,1,1}
\definecolor{ttqqqq}{rgb}{0.2,0,0}
\definecolor{qqqqff}{rgb}{0,0,1}
\definecolor{uuuuuu}{rgb}{0.27,0.27,0.27}
\begin{tikzpicture}[line cap=round,line join=round,>=triangle 45,x=1.2cm,y=1.2cm]
\clip(-2.42,-3.01) rectangle (11.7,8.11);
\draw [color=ttqqqq] (0,0)-- (8.13,0);
\draw [color=ttqqqq] (8.13,0)-- (4.07,7.04);
\draw [color=ttqqqq] (4.07,7.04)-- (0,0);
\draw (1.64,2.84)-- (5.93,3.82);
\draw (4.16,5.09) node[anchor=north west] {$I$};
\draw (4.15,7.56) node[anchor=north west] {$A$};
\draw (-0.61,0.13) node[anchor=north west] {$B$};
\draw (0.93,3.3) node[anchor=north west] {$B^+$};
\draw (6.05,4.3) node[anchor=north west] {$C^+$};
\draw (3.69,3.95) node[anchor=north west] {$\mathcal{F}$};
\draw [color=ttqqqq] (1.64,2.84)-- (5.93,3.82);
\draw [color=ttqqqq] (5.93,3.82)-- (4.07,7.04);
\draw [color=ttqqqq] (4.07,7.04)-- (1.64,2.84);
\draw [dash pattern=on 2pt off 2pt] (1.64,2.84)-- (8.13,0);
\draw (5.48,3.13) node[anchor=north west] {$m$};
\draw [dash pattern=on 2pt off 2pt] (1.64,2.84)-- (2.99,0);
\draw [dotted] (1.01,0)-- (5.53,2.44);
\draw [dotted] (5.53,2.44)-- (2.99,0);
\draw [dash pattern=on 2pt off 2pt] (1.64,2.84)-- (4.83,0);
\draw [dotted] (4.83,0)-- (5.53,2.44);
\draw [dotted] (5.53,2.44)-- (6.54,0);
\draw [dash pattern=on 2pt off 2pt] (6.54,0)-- (1.64,2.84);
\draw [line width=0.4pt,dash pattern=on 1pt off 2pt on 5pt off 4pt] (0.49,0.85)-- (7.64,0.85);
\draw (7.86,-0.1) node[anchor=north west] {$O_1=C$};
\draw (3.13,3.13) node[anchor=north west] {$J_0$};
\draw (6.69,0.6) node[anchor=north west] {$J_1$};
\draw (5.2,0.6) node[anchor=north west] {$J_2$};
\draw (3.67,0.65) node[anchor=north west] {$J_3$};
\draw (1.96,0.7) node[anchor=north west] {$J_4$};
\draw (6.4,-0.1) node[anchor=north west] {$O_2$};
\draw (4.67,-0.13) node[anchor=north west] {$O_3$};
\draw (2.77,-0.05) node[anchor=north west] {$O_4$};
\draw (0.8,-0.08) node[anchor=north west] {$O_5$};
\draw [dash pattern=on 2pt off 2pt] (1.64,2.84)-- (1.01,0);
\draw (0.5,0.77) node[anchor=north west] {$J_5$};
\draw (6.1,1.43) node[anchor=north west] {$P_2$};
\draw (4.79,1.46) node[anchor=north west] {$P_3$};
\draw (3.7,1.41) node[anchor=north west] {$P_4$};
\draw (2.49,1.46) node[anchor=north west] {$P_5$};
\begin{scriptsize}
\draw [fill=uuuuuu] (0,0) circle (1.5pt);
\draw [fill=qqqqff] (8.13,0) circle (1.5pt);
\draw [fill=ffffff] (4.07,7.04) circle (1.5pt);
\draw [fill=xdxdff] (1.64,2.84) circle (1.0pt);
\draw [fill=xdxdff] (5.93,3.82) circle (1.0pt);
\draw [fill=qqqqff] (5.53,2.44) circle (1.0pt);
\draw [fill=uuuuuu] (6.19,0.85) circle (1.0pt);
\draw [fill=uuuuuu] (6.54,0) circle (1.0pt);
\draw [fill=uuuuuu] (5.07,0.85) circle (1.0pt);
\draw [fill=uuuuuu] (4.83,0) circle (1.0pt);
\draw [fill=uuuuuu] (3.87,0.85) circle (1.0pt);
\draw [fill=uuuuuu] (2.99,0) circle (1.0pt);
\draw [fill=uuuuuu] (2.58,0.85) circle (1.0pt);
\draw [fill=uuuuuu] (1.01,0) circle (1.0pt);
\end{scriptsize}
\end{tikzpicture}
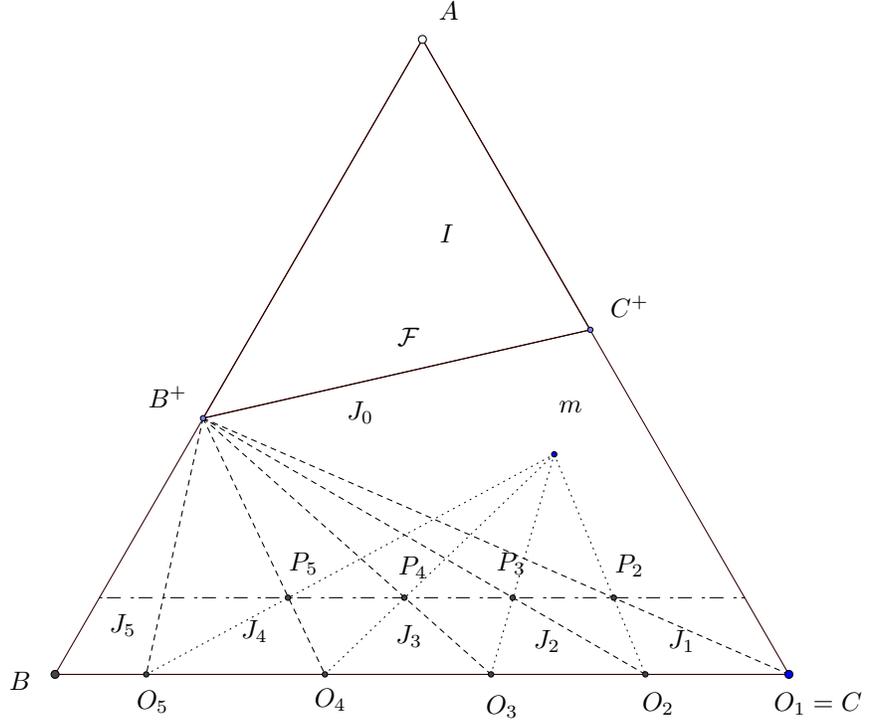
\centerline{Figure 6: The points $(O_k)_k$ and $(P_k)_k$.}

\bigskip

The next claim states that this algorithm ends in a finite  number of steps.
  \bigskip

\begin{claim}
$K<+\infty$.
\end{claim}

\begin{proofclaim}
We introduce the map $f$ from the line segment $[B,C]$ to the line $(B,C)$ as follows.
Given $X\in [B,C]$, we let $f(X)$ be the intersection of $(B^+,Y)$ with $(B,C)$, where $Y$ is the intersection of $(X,m)$ with $\calP$, see Figure 7.

\definecolor{xdxdff}{rgb}{0.49,0.49,1}
\definecolor{ffffff}{rgb}{1,1,1}
\definecolor{ttqqqq}{rgb}{0.2,0,0}
\definecolor{qqqqff}{rgb}{0,0,1}
\definecolor{uuuuuu}{rgb}{0.27,0.27,0.27}
\begin{tikzpicture}[line cap=round,line join=round,>=triangle 45,x=1.0cm,y=1.0cm]
\clip(-2.42,-3.01) rectangle (11.7,8.11);
\draw [color=ttqqqq] (0,0)-- (8.13,0);
\draw [color=ttqqqq] (8.13,0)-- (4.07,7.04);
\draw [color=ttqqqq] (4.07,7.04)-- (0,0);
\draw (1.58,2.73)-- (5.93,3.82);
\draw (4.16,5.09) node[anchor=north west] {$I$};
\draw (5.43,1.97) node[anchor=north west] {$J$};
\draw (4.15,7.56) node[anchor=north west] {$A$};
\draw (-0.61,0.13) node[anchor=north west] {$B$};
\draw (8.27,0.23) node[anchor=north west] {$C$};
\draw (0.93,3.3) node[anchor=north west] {$B^+$};
\draw (6.05,4.3) node[anchor=north west] {$C^+$};
\draw (3.57,3.89) node[anchor=north west] {$\mathcal{F}$};
\draw [color=ttqqqq] (1.58,2.73)-- (5.93,3.82);
\draw [color=ttqqqq] (5.93,3.82)-- (4.07,7.04);
\draw [color=ttqqqq] (4.07,7.04)-- (1.58,2.73);
\draw [dotted] (1.64,0)-- (5.02,3.3);
\draw [dash pattern=on 1pt off 2pt on 4pt off 4pt] (1.58,2.73)-- (6.56,2.73);
\draw [dash pattern=on 1pt off 2pt on 4pt off 4pt] (0.56,0.97)-- (7.58,0.97);
\draw [dotted] (3.2,0)-- (1.58,2.73);
\draw (1.47,0.01) node[anchor=north west] {$X$};
\draw (2.92,0.05) node[anchor=north west] {$f(X)$};
\draw (2.59,1.66) node[anchor=north west] {$Y$};
\draw (4.41,2.76) node[anchor=north west] {$Z$};
\draw (5.1,3.64) node[anchor=north west] {$m$};
\draw (6.38,1.49) node[anchor=north west] {$\mathcal{P}$};
\begin{scriptsize}
\draw [fill=uuuuuu] (0,0) circle (1.5pt);
\draw [fill=qqqqff] (8.13,0) circle (1.5pt);
\draw [fill=ffffff] (4.07,7.04) circle (1.5pt);
\draw [fill=xdxdff] (1.58,2.73) circle (1.0pt);
\draw [fill=xdxdff] (5.93,3.82) circle (1.0pt);
\draw [fill=qqqqff] (5.02,3.3) circle (1.0pt);
\draw [fill=xdxdff] (1.64,0) circle (1.0pt);
\draw [fill=uuuuuu] (2.63,0.97) circle (1.0pt);
\draw [fill=uuuuuu] (3.2,0) circle (1.0pt);
\draw [fill=uuuuuu] (4.43,2.73) circle (1.0pt);
\end{scriptsize}
\end{tikzpicture}
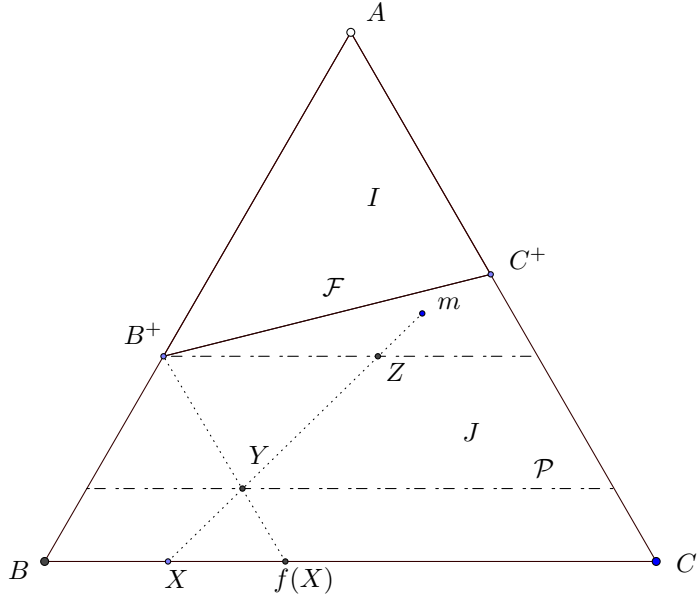
\centerline{Figure 7: The definition of $f$.}
\bigskip

Since $m$ belongs to $I\cup J_0$ and to the relative interior of $\Delta(\Omega)$, $f(X)$ is well-defined and $f(B)$ lies strictly ``to the right'' of $B$.

Observe that (by Thales Theorem), the Euclidian distance $Xf(X)$ is proportional to the distance $B^+Z$.
Hence, as $m$ moves away from $B$ towards $C$, $Xf(X)$ increases if $m\in I$, and decreases if $m\in J_0$.
In the former case, this implies that $O_{k+1}f(O_{k+1})=O_{k+1}O_k \geq Bf(B)$ for each $k$.
In the latter one, this implies that $O_kO_{k+1}$ increases with $k$. In both cases, $K<+\infty$.
\end{proofclaim}

\bigskip

For $k=1,\ldots, K-1$,  denote by $J_k$ the triangle $(B^+,O_k,O_{k+1})$ (see Figure 6),
and observe that $\phi([O_k,O_{k+1}])=[P_k,P_{k+1}]$. The belief dynamics is similar for any initial belief in $J_k$. Any $p_1\in J_k$ is first split between $B^+$ and some $q_1\in [O_k,O_{k+1}]$. In the latter case, $q_1$ is mapped to $p_2:=\phi(q_1)\in [P_k,P_{k+1}]$. The belief $p_2$ is then split between $B^+$ and
$q_2\in [O_{k-1},O_k]$, etc. The (random) belief $p_{k+1}$ in stage $k+1$ lies in $I\cup J_0$.

\bigskip

\underline{Step 3}. The function $\gamma$ is concave on $\Delta(\Omega)$.

We proceed with a series of claims.
\bigskip

\begin{claim}
The function $\gamma$ is affine on $J_k$, for every $k\leq K$.
\end{claim}

\begin{proofclaim}
We argue by induction and start with $k=0$. We denote by $\gamma(\calF)$ the constant value of $\gamma$ on $\calF$. Given $p=x C+(1-x)q\in J_0$ with $q\in [B,C]$, one has $\gamma(p)=x\gamma(C)+(1-x)\gamma(\calF)$, hence the affine property. For later use, note also that, as $p$ moves towards $[A,C^+]$ on a line parallel to $[B,C]$, the weight $x$ decreases, hence $\gamma(\cdot)$ is decreasing on such a line.

Assume now that $\gamma$ is affine on $J_{k-1}$ for some $k\geq 1$. For $p\in [O_k,O_{k+1}]$, $\gamma(p)=\delta\gamma\circ \phi(p)$. Since $\phi(p)\in J_{k-1}$, $\gamma$ is affine on $[O_k,O_{k+1}]$. Next, for
$p=x_I B^+ +x_k O_k +x_{k+1}O_{k+1}\in J_k$,
\begin{eqnarray*}
\gamma(p)&=& x_I \gamma(B^+)+(x_k +x_{k+1})\gamma\left(\frac{x_k O_k +x_{k+1}O_{k+1}}{x_k+x_{k+1}}\right)\\
&=&x_I \gamma(B^+) + x_k \gamma(O_k) +x_{k+1} \gamma(O_{k+1}).
\end{eqnarray*}
That is, $\gamma$ is affine on $J_k$.
\end{proofclaim}

\begin{claim}
The function $\gamma$ is concave on $J_k\cup J_{k+1}$ for $k=1,\ldots, K-2$.
\end{claim}

\begin{proofclaim}
We will use the following elementary observation. Let $g_1,g_2:\dR^2\to \dR$ be affine maps. Let $\calL$ be a line in $\dR^2$, such that $g_1=g_2$ on $\calL$. Let
$H_1$ and $H_2$ be the two half-spaces defined by $\calL$, and let $h$ be the map that coincides with $g_i$ on
$H_i$. Assume that for $i=1,2$, there is a point $A_i$ in the relative interior of $H_i$ such that $h$ is concave on $[A_1,A_2]$. Then $h$ is concave%
\footnote{If $g_1=g_2$ everywhere the conclusion holds trivially.
Otherwise, $g_1$ and $g_2$ coincide only on $\calL$, and then $h=\min\{g_1,g_2\}$.}
on $\dR^2$.

We prove the claim by induction.
Pick first
$\tilde p_0\in J_0\cap \calP$ and  $\tilde p_1\in J_1\cap \calP$, and let $p_*$ be the point of intersection of $\calP$ with the line $(B^+,C)$.
Under $\sigma_*$, any point $p\in [\tilde p_1,p_*]$ is split as $p=(1-x)B^++xq_J$, where $q_J\in (B,C)$. Note that $x$ does not depend on $p$, and
\begin{eqnarray*}
\gamma(p)&=& (1-x)\gamma(B^+) +x \gamma \left( \frac{p-(1-x)B^+}{x}\right)\\
&=& (1-x) \gamma(B^+) +x\delta \gamma\circ \phi\left( \frac{p-(1-x)B^+}{x}\right).
\end{eqnarray*}
As $p$ moves from $\tilde p_1$ towards $p_*$, $\displaystyle \phi\left( \frac{p-(1-x)B^+}{x}\right)$ moves from $p_*$ towards $\tilde p_0$. Hence, the derivative of $\gamma$ on $[\tilde p^1,p_*]$ is equal to $\delta \lambda$ times the derivative of $\gamma$ on $[p_*,\tilde p_0]$.\footnote{We are here identifying any point $p=y\tilde p_0+(1-y)\tilde p_1$ of $[\tilde p_1,\tilde p_0]$ with the real number $y$, and we view $\gamma$ as defined over $[0,1]$.} Since the latter derivative is negative, and $\delta\lambda <1$, $\gamma(\cdot)$ is concave on $[\tilde p_1,\tilde p_0]$. The concavity of $\gamma$ on $J_1\cup J_0$ then follows from the preliminary observation.

Assume now that $\gamma$ is concave on $J_k\cup J_{k-1}$ for some $k\geq 1$. For $p\in [O_{k+1},O_{k-1}]$,
we have $\gamma(p)=\delta\gamma\left(\phi(p)\right)$. Since $\phi(p)\in [P_{k+1},P_{k-1}]\subset J_k\cup J_{k-1}$,
the function $\gamma$ is concave on $[O_{k+1},O_k]$ hence by the preliminary observation it is also concave on $J_{k+1}\cup J_k$.
\end{proofclaim}

\begin{claim}
\label{claim:3}
The function $\gamma$ is concave on $J$.
\end{claim}

\begin{proofclaim}
Let $\tilde p_1$ and $\tilde p_2$ be given  in the relative interior of $J_{k_1}$ and $J_{k_2}$ respectively, with $k_1\leq k_2$.
Since the intersection of the line segment $[\tilde p_1,\tilde p_2]$ with each of the sets $J_{k_1}$, $J_{k_1+1},\dots, J_{k_2}$
is a line segment with a nonempty interior,
the concavity of the function $\gamma$ on each $J_k\cup J_{k+1}$ implies its concavity on $[\tilde p_1,\tilde p_2]$.

The concavity of the function $\gamma$ on $J$ follows by continuity.
\end{proofclaim}

\begin{claim}
The function $\gamma$ is concave on $\Delta(\Omega)$.
\end{claim}

\begin{proofclaim}
As above, it suffices to prove that $\gamma$ is concave on the relative interior $\stackrel{\circ}{\Delta}(\Omega)$ of $\Delta(\Omega)$.
Pick $\tilde p_1,\tilde p_2\in \stackrel{\circ}{\Delta}(\Omega)$, with $\tilde p_1\in I$ and $\tilde p_2\in J_k$ for some $k\geq 1$.%
\footnote{For other cases, the concavity of $\gamma$ on $[\tilde p_1,\tilde p_2]$ follows from either Step 1 or Claim \ref{claim:3}.}
Since $[\tilde p_1,\tilde p_2]\subset \stackrel{\circ}{\Delta}(\Omega)$,
there is a line segment $[p_*,p_{**}]\subseteq [\tilde p_1,\tilde p_2]$ with $p_*,p_{**}\in J_0$ and $p_*\neq p_{**}$.
By Step 1 the function $\gamma$ is concave on $[\tilde p_1,p_{**}]$ and by Claim \ref{claim:3} it is concave on $[p_*,\tilde p_2]$.
Therefore it is concave on $[\tilde p_1,\tilde p_2]$.
\end{proofclaim}

\bigskip\noindent
\underline{Step 4}. $d\geq 0$ on $\Delta(\Omega)$.

We start with the intuitive observation that the payoff under $\sigma_*$ is higher when starting from $\calF$ than from $J$.
\bigskip

\begin{claim}
$\gamma(p) \leq \gamma(\calF)$ for all $p\in J$.
\end{claim}

\begin{proofclaim}
This is trivial if $m\in I$, since $\gamma(B^+)$ is then equal to 1.
Assume then that $m \in J_0$.

We prove inductively that $\gamma(p) \leq \gamma(\calF)$ for all $p\in J_k$.
Note first that $\gamma(C)=\delta \gamma(\phi(C))$, so that $\gamma(C)\leq \gamma(\phi(C))$. Since
$m,C\in J_0$, we have $\phi(C)\in J_0$, hence  $\gamma(\phi(C))$ is a convex combination of $\gamma(C)$ and $\gamma(\calF)$.
This implies that $\gamma(C)\leq \gamma(\calF)$. Note next that, for $p\in J_0$, the quantity $\gamma(p)$ is a convex combination of $\gamma(C)$ and $\gamma(\calF)$,
hence $\gamma(p)\leq \gamma(\calF)$.

Assume that the conclusion holds on $J_{k-1}$ for some $k\geq 1$. For $p\in [O_{k+1},O_k]$, since $\phi(p)\in J_{k-1}$,
we have $\gamma(p)=\delta\gamma(\phi(p))\leq \gamma(\calF)$. Observe finally that for some $p\in J_k$, the quantity $\gamma(p)$ is a convex combination of $\gamma(\calF)$ and of $\gamma(q)$ for some $q\in [O_{k+1},O_k]$,
hence $\gamma(p)\leq \gamma(\calF)$ and the conclusion holds on $J_k$ as well.
\end{proofclaim}

\bigskip
We conclude with the tricky part of the proof.

\begin{claim}
For $k\geq 1$, we have $d\geq 0$ on some neighborhood of $O_{k+1}$ in $J_k$.
\end{claim}

\begin{proofclaim}
Given $\ep>0$, let $p_\ep:=\ep B^+ +(1-\ep)O_{k+1}\in J_k$. Fix $\ep>0$ small enough so that $\phi(p_\ep)\in J_{k-1}$.
Observe that both $\gamma$ and $\gamma\circ \phi$ are affine on the triangle $(p_\ep, O_{k+1},O_k)$, hence $d$ is affine on this triangle as well.
Since $d=0$ on $[O_{k+1},O_k]$ it thus suffices to prove that $d(p_\ep)\geq 0$.

We denote by $\gamma_k:\Delta(\Omega)\to \dR$ the affine map which coincides with $\gamma$ on $J_k$. Set
$q_\ep:=\ep B^+ +(1-\ep) P_{k+1}$ and observe that
\begin{equation}\label{eq17}
d(p_\ep)=\gamma(p_\ep)-\delta \gamma(\phi(p_\ep))=\gamma(p_\ep)-\delta \gamma(q_\ep) +\delta \left( \gamma(q_\ep)-\gamma(\phi(p_\ep))\right).
\end{equation}
Since $\gamma(p_\ep)=\ep \gamma(B^+)+(1-\ep)\delta\gamma(P_{k+1})$ and $\gamma(q_\ep)=\ep \gamma(B^+)+(1-\ep)\gamma(P_{k+1})$, one has
\begin{equation}
\label{equ:13.1}
\gamma(p_\ep)-\delta \gamma(q_\ep)=\ep \gamma (B^+)(1-\delta).
\end{equation}

On the other hand, since $q_\ep$ and $\phi(p_\ep)$ belong to $J_{k-1}$, one has
\begin{equation}
\label{equ:13.2}
\gamma(q_\ep)-\gamma(\phi(p_\ep))=\gamma_k(q_\ep)-\gamma_k(\phi(p_\ep))=\gamma_k(q_\ep-\phi(p_\ep))=\ep \gamma_k( B^+ -\phi(B^+)).
\end{equation}
Substituting (\ref{equ:13.1}) and (\ref{equ:13.2}) into (\ref{eq17}) one gets
\begin{equation}\label{eq17bis}
d(p_\ep)=\ep \left(  \gamma(B^+) (1-\delta) +\delta \gamma_k (B^+ -\phi(B^+))\right).
\end{equation}

Now rewrite $B^+-\phi(B^+)$ as
\begin{eqnarray*}
B^+-\phi(B^+) &=& B^+ -O_k +O_k-P_k +P_k -\phi(B^+) \\
&=& O_k -P_k +(1-\lambda) (B^+-O_k)
\end{eqnarray*}
(recall that $P_k=\phi(O_k)$).

Since all three points $O_k,P_k$ and $B^+$ belong to $J_{k-1}$, one has
\begin{eqnarray*}
\gamma_k (B^+-\phi(B^+))&=& \lambda \gamma_k(O_k) -\gamma_k(P_k)+(1-\lambda)\gamma_k(B^+)\\
&=& \lambda \gamma(O_k)-\gamma(P_k)+(1-\lambda)\gamma(B^+)\\
&=& (1-\lambda) \gamma(B^+) -(1-\lambda\delta) \gamma(P_k).
\end{eqnarray*}
Plugging into (\ref{eq17bis}), one finally gets
\[d(p_\ep)=\ep(1-\lambda \delta)\left( \gamma(B^+)-\delta \gamma(P_k)\right),\]
which is nonnegative by Claim 1.
\end{proofclaim}

\bigskip

We now conclude the proof of Step 4. Let $p\in J_k$ be given. Since
$\gamma$ is affine on $J_k$ and concave on $\Delta(\Omega)$, the function $d$ is convex on $J_k$. Since
$d(O_{k+1})=0$ and
$d\geq 0$ in a neighborhood of $O_{k+1}$ (in $J_k$), $d$ is nonnegative on the entire line segment $[O_{k+1},p]$.

\end{document}